\documentclass[11pt]{article}

\usepackage[a4paper,left=2cm,right=2cm,top=2cm,bottom=2cm]{geometry}
\usepackage[utf8]{inputenc}
\usepackage{natbib}
\usepackage{amsmath,accents}
\usepackage{amsthm}
\usepackage{amssymb}
\usepackage{amsfonts}

\usepackage[T1]{fontenc}
\usepackage{cite}
\usepackage{graphicx,subfigure}
\usepackage{chngcntr}
\usepackage{xcolor}
\usepackage{hyperref}
\usepackage[all]{hypcap} % Link refers to the figure not the label

\counterwithin{equation}{section}
\hypersetup{
  colorlinks   = true, %Colours links instead of boxes
  urlcolor     = red, %Colour for external hyperlinks
  linkcolor    = red, %Colour of internal links
  citecolor   = red %Colour of citations
}
\usepackage{todonotes}

\definecolor{colorJRblue}{rgb}{0.,0.,1.}%67}
\definecolor{colorJRred}{rgb}{1.,0.,0.}%67}
\definecolor{colorJRgreen}{rgb}{0.,0.6,0.}%67}

\def\blue{}%{\color{colorJRblue}}

\theoremstyle{definition}
\newtheorem{rem}{Remark}[section]
\theoremstyle{plain}  
\newtheorem{thm}[rem]{Theorem}

\newcommand{\R}{\mathbb{R}}
\newcommand{\Z}{\mathbb{Z}}
\newcommand{\C}{\mathbb{C}}

\newcommand{\N}{\mathbb{N}}
\newcommand{\T}{\mathbb{T}}

\def\v{\boldsymbol{v}} % 3D velocity
\def\vh{\boldsymbol{u}} % 2D velocity
\def\u{\boldsymbol{u}}

\def\w{\boldsymbol{w}}
\def\k{\boldsymbol{k}}
\def\e{\boldsymbol{e}}

\def\a{\boldsymbol{a}}
\def\x{\boldsymbol{x}}
\def\2{\boldsymbol{e_2}}
\def\3{\boldsymbol{e_3}}
\def\c{\boldsymbol{c}}
\def\p{\partial}
\def\b{\mathtt{b}}
\def\tb{\tilde{\b}}
\def\eps{\varepsilon}
\def\calL{\mathcal{L}}
\def\rmi{\mathrm{i}}
\def\d{\mathrm{d}}
\def\sc{k}
\def\diag{\mathrm{diag}}
\def\div{\operatorname{div}}

\def\Id{\mathrm{Id}}
\def\B{\boldsymbol{B}}

\def\Re{\mathrm{Re}}

\def\disp{D} % dispersion relation

\def\gradh{\nabla} % 2D horizontal gradient
\def\Laplaceh{\Delta} % 2D horizontal Laplacian
\def\divh{\nabla\cdot} % 2D divergence
\def\grad{{\nabla_3}} % 3D gradient
\def\Laplace{{\Delta_3}} % 3D Laplacian

\def\cref{\eqref}
\def\Cref{\S\ref}
\def\F{\boldsymbol{F}}

\def\lb{{C}}
\def\dif{\mathrm{d}}
\def\Lspace{L}
\def\Hspace{H}
\def\kx{k_x}
\def\ky{k_y}

\def\crit{{\rm c}}
\def\tee{\widetilde{\mathbf{e}}}
\def\E{\mathbf{E}}
\def\tE{\widetilde{\mathbf{E}}}
\def\rme{\mathrm{e}}
\def\kcx{k'_x}
\def\kcy{k'_y}
\def\ee{\mathbf{e}}

\def\calO{\mathcal O}
\def\calR{\mathcal R}
\def\s{{\rm s}}
\def\ns{{\rm ns}}
\def\zz{\chi}

\def\pd{\underaccent{\bar}{\partial}}
\def\nablad{\underaccent{\bar}{\nabla}}
\def\Deltad{\underaccent{\bar}{\Delta}}
\def\nabladh{\nablad} % 2D horizontal discrete gradient
\def\Deltadh{\Deltad} % 2D horizontal discrete Laplacian
\def\Bd{\underaccent{\bar}{\boldsymbol{B}}}
\def\Vd{\underaccent{\bar}{\boldsymbol{V}}}
\def\Fd{\underaccent{\bar}{\boldsymbol{F}}}

\title{Rotating convection {\blue and flows} with horizontal kinetic energy backscatter}

\author{Paul Holst\thanks{University of Bremen, Department 3 - Mathematics, 28359 Bremen, Germany
  ({pholst@uni-bremen.de}).}
\and Jens D.M. Rademacher\thanks{Universit\"at Hamburg, MIN faculty, Department of Mathematics, Germany ({jens.rademacher@uni-hamburg.de}).} 
\and Jichen Yang\thanks{Harbin Engineering University, School of Mathematical Sciences, 150001 Harbin, China
  ({jichen.yang@hrbeu.edu.cn}).}
}

\date{\today}

\begin{document}

\maketitle

\tableofcontents

\begin{abstract}
Numerical simulations of large scale geophysical flows typically require unphysically strong dissipation for numerical stability. Towards energetic balance various schemes have been devised to re-inject this energy, in particular by horizontal kinetic energy backscatter. In a set of papers, some of the authors have studied this scheme through its continuum formulation with momentum equations augmented by a backscatter operator, e.g. in rotating Boussinesq and shallow water equations. Here we review the main results about the impact of backscatter on certain flows and waves, including some barotropic, parallel and Kolmogorow flows, as well as internal gravity waves and geostrophic equilibria. We particularly focus on the possible accumulation of injected energy in explicit medium scale plane waves, which then grow exponentially and unboundedly, or yield bifurcations in the presence of bottom drag. Beyond the review, we introduce the rotating 2D Euler equations with backscatter as a guiding example. For this we prove the new result that unbounded growth is a stable phenomenon occurring in open sets of phase space. We also briefly consider the primitive equations with backscatter and outline global well-posedness results.
\end{abstract}

\section{Introduction}\label{s:intro}
Two particular challenges for simulating climate models and geophysical flows on large scales over long times are numerical instabilities and the necessarily coarse resolution. 
Such instabilities are typically suppressed by unphysically strong dissipation and viscosity, and so-called subgrid parameterizations are used to compensate unresolved processes. 
These include energy transfer from small to large scales, which has been consistently found in high resolution simulations of geophysical flows in addition to the forward dissipation cascade, e.g., \citep{juricke2020kinematic,DJKO2019}. See also \citep{AlexBifer2018}. 

This motivates to re-inject some portion of energy at large scales, and is implemented in so-called horizontal kinetic energy backscatter schemes, e.g., \citep{JH2014, ZuritaEtAl2015,KJC2018,JansenEtAl2019, DJKO2019,JDKO2019, juricke2020kinematic,Perezhogin20}. 
Numerical studies in the cited references find that these schemes appear to provide energy at the right place, and for instance successfully excite otherwise damped eddy activity. 

In order to gain insight into the impact of such schemes from a mathematical viewpoint, we view these as a modification of the fluid momentum equations on the continuum level, i.e., as partial differential equations. This admits to analytically study consequences of backscatter in a way that is currently inaccessible for the discrete form. Backscatter then essentially appears through a linear operator composed of negative Laplacian and bi-Laplacian, whose Fourier symbol with wavenumber $k$ is of the form
\begin{align}\label{e:dispBack}
\lambda = b k^2-d k^4, \; b,d\geq 0. 
\end{align}
This operator combines negative viscosity with stabilizing hyperviscosity and, for $d=b>0$, is familiar from the Kuramoto-Sivashinsky equation (KS). The one-dimensional KS famously admits a global attractor and is a paradigm for chaos in a partial differential equation, e.g.\ \citep{NST85, Kalogirou15}; in two and three space dimensions the study of KS has regained attention in recent years \citep{Kalogirou15,AM2019,AM2021,CZDFM2021,FM22,FSW22,KM2023}. These results suggest that backscatter can profoundly impact the dynamics. 

\medskip
The continuum perspective has been pursued in \citep{PRY22, PRY2023} and here we review the main results. In particular, we discuss families of explicit flows in which the re-injected energy can get trapped, leading to exponential and unbounded growth. And we review bifurcations that are induced by the presence of nonlinear bottom drag, which implies selections of geostrophic equilibria and shallow water gravity waves. 

Beyond this review, we introduce the rotating 2D Euler equations augmented with backscatter 
as a guideline, which also appears as a subsystem in all models, and where the unbounded growth is particularly transparent. Moreover, as a new result, we prove for this system that unbounded growth appears in open sets of phase space for moderate ratio of the backscatter coefficients $b/d$.  This can be interpreted as stability of unbounded growth from explicit solutions and we informally state it as follows. 

\begin{thm} \label{t:stab2DNSE}
Consider the 2D Euler equations with constant backscatter posed on the torus. For an open set of backscatter parameters there are stably unboundedly growing solutions $\u_{\ast}(t)$, i.e., any solution $\vh(t)$ with initial $\vh(0)$ sufficiently close to $\u_{\ast}(0)$ also grows unboundedly. Moreover, the growth of $\vh$ is exponential with rate arbitrarily close to that of $\u_{\ast}$.
\end{thm}

We also briefly consider the primitive equations with backscatter and, as a study of backscatter from another viewpoint, we outline well-posedness results that will appear in full detail elsewhere. Indeed, backscatter with small $b,d$ can be viewed as a particular regularization of the inviscid equations. Notably, the occurrence of unbounded growth precludes a compact global attractor and the well-posedness problem differs from that of the Navier-Stokes or primitive equations with viscosity. We note that the differentiated KS equation has the form of the momentum equation of Euler augmented by backscatter, but without pressure. However, it concerns only curl-free velocities of the form $\v=\nabla u$, and the known unboundedly growing solutions all have non-zero curl. It also turns out that the incompressibility helps for obtaining well-posedness. 
{\blue We refer to such 3D or shallow water flows as convective in the sense of fluid motion driven by density differences or just gravity,
%\JY{By this definition, convection does not occur in the constant density fluid models (e.g. shallow water equations), right?\JR{Right... I have added gravity. However, the 2D NSE has neither so that perhaps adding "flows" in the title is a good idea.}}
 not in the sense of atmospheric convection.} 

\bigskip
This chapter is organized as follows. In \S\ref{s:modeling} we briefly describe the dynamic and kinematic backscatter schemes and the corresponding modification of the continuum form for several PDE. We also discuss the spectrum for the resulting modified linear rotating shallow water equations in some detail. 
Section \ref{s:PlaWave} concerns the impact of backscatter on certain waves and flows, and the occurrence of unbounded growth. In \S\ref{s:stabgrow} we prove the stability of unbounded growth for the 2D Euler equation with backscatter. Section \ref{s:bottom} is devoted to review the bifurcations induced by backscatter combined with linear and nonlinear bottom drag in the rotating shallow water equations. The chapter ends with a discussion and outlook in \S\ref{s:discussion}.

\section{Backscatter modeling}\label{s:modeling}

Backscatter is practically used in a variety of models. A rather complex case is its implementation in the 
primitive equations ocean model, for which the modified horizontal momentum equation takes the form
\begin{align}\label{e:primmom}
\pd_t \vh + (\vh \cdot \nabladh)\vh + w \pd_z \vh + f \vh^\perp + 
\tfrac{1}{\rho_0}\nabladh p     
= \Bd(\vh) + \pd_z(A_\nu \pd_z\vh).
\end{align} 
with 2D horizontal velocity field $\vh = (u,v)$, vertical velocity $w$, vertical coordinate $z$, Coriolis parameter $f$,  reference density $\rho_0>0$, and discretized (indicated by underbar) partial derivative $\pd_z$ as well as discretized horizontal gradient $\nabladh := (\pd_x,\pd_y)$. 
The specific discretization of the differential operators is not relevant for us; we refer to, e.g., \citep{JDKO2019}. The vertical viscosity coefficient $A_\nu$ is specified by a vertical mixing parametrization  that we do not further discuss here. A combined viscosity-backscatter parameterization is included through $\Bd(\vh)$. In \eqref{e:primmom} and the following, for 2D vectors $(a_1,a_2)$ we set $(a_1,a_2)^\perp = (-a_2,a_1)$. 

As outlined before, numerical stability requires unphysically strong viscosity, so that, torwards energy consistency,  a source of energy is required. Kinetic energy backscatter aims to re-inject the energy that is excessively dissipated, but on a larger spatial scale and at a controlled rate. 
In all cases $\Bd$ (and $A_\nu)$ essentially scale with a grid scale measure so that these formally vanish in a naive continuum limit. However, practically used discretizations for climate or ocean simulations are coarse and relevant subgrid dynamics needs to be parameterized.

\medskip 
Given a chosen viscosity parameterization $\Vd(\vh)$ it is natural to introduce backscatter through an additional term $\Bd_b(\vh)$ via
\[
\Bd(\vh) = \Vd(\vh) + \Bd_b(\vh).
\] 
For numerical reasons independent of backscatter, a common choice of $\Vd$ has a grid scale dependent biharmonic or hyperviscous nature, e.g. \citep{JDKO2019,KJC2018}. Then backscatter can utilize negative viscosity as in 
\[
\Bd_b(\vh) = -\nu_b \Deltadh \vh,
\] 
where $\Deltadh:=\pd_x^2+\pd_y^2$ is the discretized horizontal Laplace operator and may also include a local averaging/smoothing filter. The backscatter coefficients are chosen in relation to the amount of unresolved or overdissipated kinetic energy. 
We briefly outline two approaches and refer to, e.g., \citep{JDKO2019,juricke2020kinematic} for more detailed discussions. In so-called \emph{dynamic backscatter} an unresolved subgrid energy $e$ is introduced that determines $\nu_b$ in the form
\[
\nu_b = C \sqrt{\max(e,0)}\geq 0,
\]
with constant $C$, and $e$ is itself determined via 
\[
\pd_t e = - c_{\rm dis} \dot E_d - \dot E_b - \nablad\cdot(\nu^C \nablad e). 
\]
Here $\dot E_d$ denotes the resolved energy dissipation rate, $c_{\rm dis}$ a constant factor, $\dot E_b$ denotes the injection rate of unresolved kinetic energy to the resolved flow and $\nu^C$ realizes grid scale dependent diffusion of unresolved kinetic energy. 
Different choices for $\dot E_d$ are used, inspired by a continuum form of $\Vd$ and its energy dissipation. Similarly, $\dot E_b$ is determined as a sink of unresolved kinetic energy in terms of $\Bd_b$. However, some averaging or filtering may be needed. In this way, dynamic backscatter extends the momentum equations by a subgrid energy equation that is fully coupled to it. 

\medskip
The simpler \emph{kinematic backscatter} scheme has been proposed in \citep{juricke2020kinematic} in order to reduce computational load, and also to assess the requirements for and impact of backscatter further. In this scheme the specific choice of $\Bd$ combines viscous dissipation and backscatter. It takes the form
\[
\Bd = \Vd_d-\tilde\alpha \Fd \Vd_d,
\]
where $\Vd_d$ is a discrete harmonic viscous operator, and $\Fd$ a near-local filter operator. Here $\Vd_d$, $\tilde\alpha$ and $\Fd$ are chosen such that $-\tilde\alpha \Fd \Vd_d$ acts as an anti-dissipative term and $\Vd_d$ takes the role of the biharmonic dissipation $\Vd$ in dynamic backscatter. This exploits the discrete nature of the operators, where the dispersion relations include higher order terms than the continuum differential operators. This can be illustrated explicitly in 1D upon expanding the dispersion relation for $\Vd_d = \nablad_+\nablad_-$ and 
$\Fd = \nablad_+^2+\nablad_-^2$ with left and right finite difference operators $\nablad_\pm$, cf. \citep{juricke2020kinematic}. 
In particular, a straightforward computation shows that a positive quadratic term occurs for $\tilde\alpha>1/4$, which thus forms a threshold for anti-viscous energy injection. At next order a quartic term with coefficient $d$ corresponds to biharmonic hyperviscosity. To leading order we thus obtain the growth rates through the Fourier symbol 
\eqref{e:dispBack}, 
where $b=\tilde\alpha-1/4$ and the values of $k$ relate to the grid scale. In this computation $\tilde\alpha$ is tacitly taken constant and we note that artificially fixing $\nu_b$ in dynamic backscatter also gives \eqref{e:dispBack}. 
{\blue We refer to \citep{Grooms} for a discussion backscatter parameterizations  for the quasi-geostrophic and primitive equations, which also considers backscatter terms in the buoyancy equation.} %\JRx{How about this?}\JYx{Ok.}

\subsection{PDE with backscatter}\label{s:PDE}

Motivated by the discrete modeling, we consider continuum formulations that admit an analysis and thus help understand the impact of backscatter parameterization. 
The natural continuum model for $\Bd$ in terms of the usual continuum horizontal gradient and Laplace operators $\nabla:= (\p_x,\p_y)$, $\Delta:= \p_x^2+\p_y^2$, is then an operator for the horizontal momentum equation given by
\[
\B =  -(b\Delta+d\Delta^2).
\]
We refer to $b,d>0$ as backscatter parameters, and note these are constant in the simplest situation. 
When posed on $\R^2$, Fourier transforming $\x$ with wave vector $\k\in\R^2$ and {\blue Euclidean} norm $k=|\k|$ then gives the transformed $\B$ as
\[
\hat \B = b k^2-d k^4.
\]
Hence, its spectrum and that of $\B$ (on e.g. $L_2$-based spaces) is directly given by the Fourier symbol in \eqref{e:dispBack}, but for $k\in\R$. 
We also consider the spatial domain $\T^2=\R^2/2\pi\Z^2$, i.e. a square domain with periodic boundary conditions rescaled to sides of length $2\pi$. In this case the spectrum is determined by the discrete subset $k\in \{\sqrt{n^2+m^2}: n,m\in\N_0\}$ since $\exp(\rmi \k\cdot\x)$ is periodic on $\T^2$ exactly for $\k\in \Z^2$. See Figure~\ref{f:specback}. 
Hence, the spectrum on $\R^2$ is unstable for $b>0$, and on $\T^2$ if $b>d$. Moreover, on $\R^2$ there is a band of unstable spectrum from $k\in (0,\sqrt{b/d})$ while on $\T^2$ there is, e.g., a single unstable eigenvalue if $b\in (d,2d)$. The latter follows since the admissible values of $k$ in increasing order are $0,1,\sqrt{2}, \ldots$ and $k=\sqrt{2}$ in \eqref{e:dispBack} gives $\lambda = 2b-4d$.

\begin{figure}[t!]
\centering
\subfigure[]{\includegraphics[trim = 5.5cm 8.5cm 6cm 8.8cm clip, height=4.5cm]{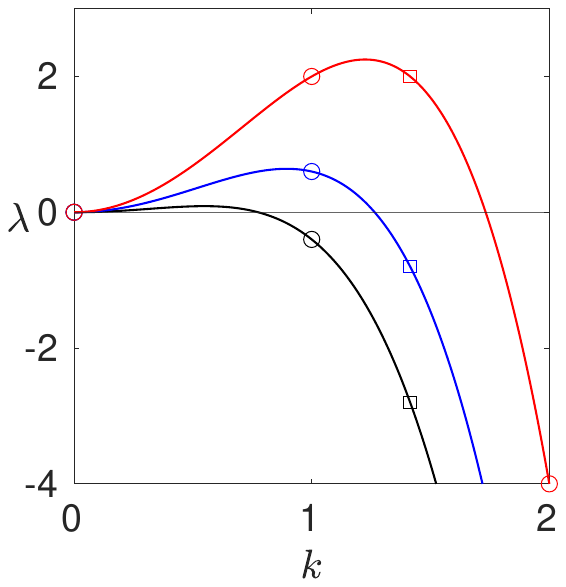}}
\hfil
\subfigure[]{\includegraphics[trim = 5.5cm 8.5cm 6cm 8.8cm, clip, height=4.5cm]{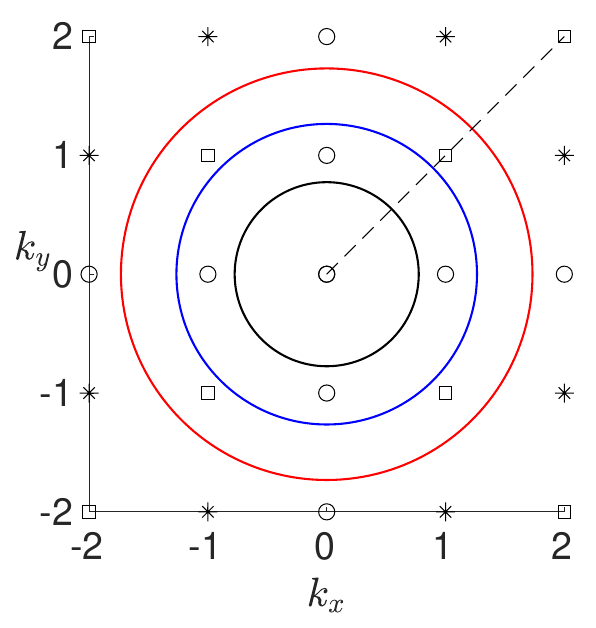}}
\caption{Samples of spectra of the backscatter operator $\B$ with wave vector $\k=(k_x,k_y)\in\R^2$ in (b), wave length $k=|\k|$ ({\blue Euclidean} norm) and growth rate $\lambda$ in (a). Parameters are $d=1$ and $b=0.6$ (black), $b=1.6$ (blue) and $b=3$ (red). Spectra of $\B$ for domain $\R^2$ are curves in (a) and disks in (b), and for domain $\T^2$ points with wave vectors on the axes (circles) and in $(1,1)$-direction (squares).  The unstable spectrum ($\lambda>0$) in (b) is located inside the disks with the respective radii $\sqrt{b/d}$, except for the neutrally stable spectrum at the origin for any choice of parameters.}
\label{f:specback}
\end{figure}

\medskip
We start with an illustrating example of a resulting PDE that is mainly used for mathematical purposes: the rotating 2D Euler equations with isotropic backscatter for 2D horizontal velocity field $\vh=(u,v)$ 
given by 
\begin{subequations}\label{eq:2DNS}
\begin{align}
\frac{\p\vh}{\p t}+(\vh\cdot\nabla)\vh+f\vh^\perp+\nabla p &= -(b\Delta+d\Delta^2) \vh\label{eq:2DNSa}\\
\nabla\cdot\vh&= 0,\label{eq:2DNSb}
\end{align}
\end{subequations}
on a horizontal domain $\x = (x,y) \in \Omega_h$, where we consider $\Omega_h=\R^2$ or $\Omega_h=\T^2$. 

An immediate impact of backscatter is revealed by the energy equation, obtained by testing \eqref{eq:2DNSa} with $\vh$, which reads
\begin{align}\label{e:2DNS_en}
\frac 1 2 \frac{\d}{\d t} \|\vh\|_2^2 = b\|\nabla \vh\|_2^2 - d\|\Delta \vh\|_2^2,
\end{align}
and which foreshadows the possible persistent energy injection that will appear in \S\ref{s:PlaWave}, \S\ref{s:stabgrow}. 
The role of scales in this process is more transparent when studying the spectrum of the linear terms of the system, i.e., dropping the transport nonlinearity, which coincides with the linearization in the trivial flow and thus determines its spectral and linear stability. For $f=0$ projecting \eqref{eq:2DNSa} onto divergence free vector fields with the Leray projection $P$ (cf.\ \S\ref{s:stabgrow}), gives the linear equation $\p_t \vh=\B\vh$, which implies the spectral and stability properties noted above. 
For $f\neq 0$ these properties are in fact the same, only eigenvectors may differ. This can be seen from the Fourier transform of the resulting linear operator $L$, which involves that of the Leray projection and is given by 
\begin{align}\label{e:Lin2DNS}
\hat L= \hat \B - f \hat P\begin{pmatrix}0&-1\\1&0\end{pmatrix}, \quad 
\hat P = \mathrm{Id} - \frac{1}{\kx^2+ \ky^2}\begin{pmatrix}\ky^2 & -\kx \ky\\ -\kx \ky & \kx^2\end{pmatrix}
\end{align}
where $\k=(\kx,\ky)\in\R^2\setminus\{0\}$. Since $\hat L-\hat \B$ has zero trace and determinant, $\hat L$ and $\hat \B$ have the same trace and determinant, and thus the same eigenvalues. Hence, $f$ only impacts the eigenvectors, but those in the kernel of $\hat L-\hat \B$ are $f$-independent, in particular for $\kx=0$ or $\ky=0$. See \S\ref{s:PlaWave}.

\medskip
A rather complex model from the geophysical application background are the 3D \emph{primitive equations} of the ocean in an $f$-plane approximation. Augmented with backscatter these are then given by
\begin{subequations}\label{eq: primEq}
\begin{align}
\frac{\p\vh}{\p t} + (\vh \cdot \gradh)\vh + w \frac{\p\vh}{\p z} + f \vh^\perp + \tfrac{1}{\rho_0}\gradh p     
&= 
\B \,\vh + 
\nu_v \frac{\p^2\vh}{\p z^2} + \boldsymbol{f}_h,  \\ 
\frac{\p T}{\p t} + \v \cdot \grad T &= \mu \Laplace T + f_3,  \\
\grad\cdot\v &=0,  \\
\frac{\p p}{\p z} &= - \rho g,  \\
\rho &= \rho_0(1 + aT), 
\end{align}
\end{subequations}
with three-dimensional velocity field $\v=(\vh,w)$ for $t \geq 0$, $\x \in\Omega_h \times [-H,0]$ with ocean depth $H>0$, the 3D gradient and Laplace operators $\grad:=(\p_x,\p_y,\p_z)$ and $\Laplace:=\p_x^2+\p_y^2+\p_z^2$, the pressure, temperature and density fields $p(t,\x)$, $T(t,\x)$, $\rho(t,\x) \in \R$, the vertical viscosity $\nu_v \geq 0$ as well as thermal diffusivity $\mu \geq 0 $, some potential forcing terms $\boldsymbol{f}_h(t,\x) \in \R^2$, $f_3(t,\x) \in \R$, and 
a constitutive parameter $a \in \R$. 
Possible boundary conditions for \eqref{eq: primEq}, beyond horizontal periodicity of $\v,p,T$ in case $\Omega_h=\T^2$, are 
\begin{subequations}\label{eq: primEqBC}
\begin{align}
&\tfrac{\p \vh}{\p z}|_{z = 0} = \boldsymbol{\tau}, && \tfrac{\p T}{\p z}|_{z = 0} = -\gamma T|_{z = 0} + T^{\ast}, && w|_{z = 0} = 0, \\
&\tfrac{\p \vh}{\p z}|_{z = -H} = \beta \vh|_{z = -H}, && \tfrac{\p T}{\p z}|_{z = -H} = 0, && w|_{z = -H} = 0. 
\end{align}
\end{subequations}
with possibly non-constant wind stress $\boldsymbol{\tau}$ and constants $\gamma,T^{\ast}, \beta$. 
Often $\beta=0$, e.g., \citep{CT07}, but also $\beta>0$ is used \citep[eq.\ (2.45)]{OWEBook}. 

It turns out that it is insightful to allow the backscatter to be anisotropic in the form
\[
\B = -\diag(d_1\Laplaceh^2+b_1\Laplaceh, d_2\Laplaceh^2+b_2\Laplaceh),
\]
where $d_j >0$ and $b_j>0$, $j=1,2$ can (slightly) differ, respectively. The motivation for this is to account for an effective anisotropy of the numerical backscatter scheme due to the combination of an anisotropic grid, spatially anisotropic coefficients and the application of filtering, see e.g.\ \citep{DJKO2019}. 
We may also include space-time dependence of the backscatter parameters, which mostly concerns $b_1,b_2$.

\medskip
We present two more PDE for which backscatter induced selection and growth phenomena will be reported in \S\ref{s:PlaWave} and \S\ref{s:bottom}. The \emph{rotating Boussinesq equations} forms another rather complex standard geophysical model without hydrostatic approximation. Augmented with horizontal backscatter operator $\B$ and using a scalar buoyancy variable $\b$ these read
\begin{subequations}\label{eq: introB}
\begin{align}
\frac{\p\v}{\p t}+(\v\cdot\grad)\v+f\3\times\v+\grad p-\3 \b &= -\diag(d_1\Laplace+b_1,d_2\Laplace+b_2,-\nu_v)\Laplace \v\label{eq: introBa} \\
\grad\cdot\v&= 0\label{eq: introBb}\\
\frac{\p \b}{\p t}+\v\cdot\grad\b+N^2w &= \mu\Laplace \b\,,\label{eq: introBc}
\end{align}
\end{subequations}
with velocity field $\v(t,\x)\in\R^3$ for $t\geq 0$, $\x\in\Omega_h\times I$ for an interval $I$, vertical unit vector $\3=(0,0,1)$. Boundary conditions can be analogous to the primitive equations, and we will also consider the full space with $I=\R$ or a periodic box. As usual, the buoyancy considered here is of the form $\b(t,\x)=-g(\rho(t,\x)-\overline{\rho}(z))/\rho_0\in\R$ with  reference density field $\overline{\rho}(z)$ depending on the vertical space direction $z$ only. 
Then $N^2=-(g/\rho_0)\dif\overline{\rho}/\dif z$ is the Brunt-V\"ais\"al\"a frequency for stable stratification $\dif\overline{\rho}/\dif z<0$. 

%\JY{In the recent paper by Grooms \citep{Grooms}, the backscatter terms in the buoyancy's equations are considered.\JR{Nice! We should use the opportunity to point out to this paper. Not sure where is the best place, but I will think about it.}}

\medskip
An intermediate complexity model from the geophysical application background are the \emph{rotating shallow water equations}. In addition to backscatter we add a bottom drag term $\F$ as a representative additional dissipation mechanism so that the equations take the form
\begin{subequations}\label{e:sw}
\begin{align}
\frac{\partial\vh}{\partial t} + (\vh\cdot\nabla)\vh +f\vh^{\perp} + g\nabla\eta  &= -\diag(d_1\Laplaceh + b_1, d_2\Laplaceh +b_2)\Laplaceh\vh
+ \F(\vh,\eta),\label{e:swa}
\\
\frac{\partial\eta}{\partial t} + (\vh\cdot\nabla)\eta 
 & = -(H_0+\eta)\nabla\cdot\vh, \label{e:swb}
\end{align}
\end{subequations}
where $\vh=(u,v)$, $\vh(t,\x)\in\R^2$ is the horizontal velocity field on $\x = (x,y) \in\Omega_h$ at time $t\geq 0$, $\eta=\eta(t,\x)\in\R$ is the deviation of the fluid layer from the mean fluid depth $H_0>0$, so $\eta$ has zero mean and the thickness of the fluid is $H = H_0+\eta$ (with flat bottom); $g>0$ is the gravity acceleration.

%In practice, {\blue as the ratio bewteen layer thicknesses and mesh-point distances,}\JRx{I do not understand the sentence -- there is something wrong with the grammar...} 
In practice, {\blue instead of 3D PDE models, sometimes a few vertically} coupled layers are used, each a shallow water equation.
%\JY{A question out of the scope of this paper: does multi-layer models with backscatter make sense when the layer thickness is smaller than the horizontal mesh point distance? Yet, if the mesh point distance is small enough, then the backscatter (both biharmonic and harmonic terms) may not be so useful since it was initially designed for the coarse resolution of ocean models.\JR{You are right that in general this does not seem to make sense, but as written in the next sentence, following the reviewer: the distance itself is maybe not a good measure, but one should consider $f/N$. This is just about ratios, not absolute smallness, so it can well be coarse.}} 
{\blue Indeed, while in discretized models the layer thickness is usually smaller than the horizontal mesh point distance, their ratio should be $f/N$, so that vertical discretization is often effectively more coarse.} %\JRx{I did not understand the previous modfication, which is the latex code. Is this ok?}
 In that case, bottom drag is only present in the lowest layer.  For the bottom drag, we consider the form 
\begin{equation}\label{e:drag}
\F(\vh,\eta) = - \frac {C + Q|\vh|} {H_0+\eta}\vh, \quad |\vh| = \sqrt{u^2 + v^2}, 
\end{equation}
where $C\geq 0$ controls the dissipation that is linear in the velocity, and $Q\geq 0$ the quadratic part, cf.\ e.g.\ \citep{AS2008}; quadratic drag combined with dynamic backscatter is used in \citep{KJC2018,DJKO2019}. 

We point out that mass is conserved independent of backscatter since \cref{e:swb} is unaffected, which justifies the assumed zero mean of $\eta$. In general the presence of a conservation law can impact the dynamics at onset of instability. In a fluid-type context, this has been studied in the presence of a reflection symmetry or Galilean invariance in e.g.\  \citep{MC2000,MC2000b}, but in \cref{e:sw} the gradient terms and $\F$ or $f\neq 0$, break these structures.

We remark that backscatter can also be viewed as a (non-standard) viscous regularization of the momentum equations, which might lead to interesting selection of weak solutions to, e.g., the Euler equations. We do not pursue this further here.

\subsubsection{Spectra for augmented rotating shallow water model}\label{s:spec} 
As for the 2D NSE, it is instructive to study the impact of backscatter on the spectrum through the linearization of the PDE models. Here we present some details from \citep{PRY2023} for \cref{e:sw} linearized in the trivial steady flow $(\vh,\eta) = (0,0)$, i.e. simply dropping the nonlinear terms. This gives the linear operator
\begin{equation}\label{e:linop}
\calL =  
\begin{pmatrix}
- d_1\Delta^2 - b_1\Delta - C/H_0 & f & -g\partial_x  \\
-f & - d_2\Delta^2 - b_2\Delta - C/H_0 & -g\partial_y \\
-H_0\partial_x & -H_0\partial_y & 0
\end{pmatrix}.
\end{equation}
The spectrum of $\calL$ with spatial domain $\R^2$ (realized e.g.\ on $\Hspace^4(\R^2)\times\Hspace^4(\R^2)\times\Hspace^1(\R^2) \subset (\Lspace_2(\R^2))^3$) consists of the roots $\lambda$ of the dispersion relation 
\begin{equation}\label{e:disp}
\disp(\lambda,\k) := \det(\lambda\Id - \hat\calL) = \lambda^3 + a_1 \lambda^2 + a_2 \lambda + a_3 = 0,
\end{equation}
where $\hat\calL$ is the Fourier transform of $\calL$, $\k = (\kx,\ky)\in\R^2$ are the wave vectors,  and 
\begin{align*}
a_1 \ &:=\ (d_1 + d_2) |\k|^4 - (b_1 + b_2) |\k|^2 + 2C/H_0, \\
a_2 \ &:=\ (d_1 |\k|^4 - b_1 |\k|^2 + C/H_0) (d_2 |\k|^4 - b_2 |\k|^2 + C/H_0) + g H_0 |\k|^2 + f^2, \\
a_3 \ &:=\  g H_0 |\k|^2 \left((d_1|\k|^2-b_1)k_y^2+(d_2|\k|^2-b_2)k_x^2 + C/H_0 \right).
\end{align*}
We write the natural continuous selections of solutions as $\lambda(\k)$. 
If the spatial domain is $\T^2$, then as for \eqref{eq:2DNS} the spectrum (e.g.\ on the corresponding spaces) is determined from the discrete subset of $\k$ where $\exp(\rmi \k\x)$ is $2\pi$-periodic. 
The sign of $\Re(\lambda)$ for the solutions to \cref{e:disp} determines the spectral stability of $(\vh,\eta)=(0,0)$ with respect to perturbations with wave vector $\k$. 

Concerning the structure of the spectrum, on the one hand, $\lambda=0$ is always solution at $\k=0$ with  eigenmode/-function $(\vh,\eta) =(0,1)$, which modifies the total fluid mass and shifts within the family of trivial steady flows $(\vh,\eta)\equiv(0,h)$, $h\in\R$. 
Hence, this mode is absent when imposing fixed mean fluid depth $H_0$, i.e. $\eta$ with zero mean. 
On the other hand, the hyperviscous terms for $d_1d_2\neq 0$ generate solutions $\lambda_\infty(\k)$ with $\Re(\lambda_\infty(\k))\to 0$ as $|\k|\to \infty$. See Figure \ref{f:specinf}. Indeed, setting $\k=r^{-1}\bar\k$ with $|\bar\k|=1$ gives 
\begin{equation}\label{e:speczero}
\lim_{r\to 0} r^8\disp(\lambda, r^{-1}k_x,r^{-1}k_y) = d_1d_2\lambda,
\end{equation}
which yields the solution $\lambda=0$ at $r=0$, i.e.\ $|\k|=\infty$ and $\lambda_\infty(\k)$ for sufficiently large $|\k|$ via the implicit function theorem. Hence, $\calL$ does not have a spectral gap to the imaginary axis for \cref{e:sw} posed on $\Omega_h=\R^2$ and on $\Omega_h=\T^2$.  As a consequence, the standard approach to center manifolds cannot be applied in order to reduce the bifurcation problem from a trivial flow to a finite dimensional ODE. We note that in practice, backscatter is applied to a numerically discretized situation, which effectively cuts the available wave vectors at some large $|\k|$. While a spectral gap is retained in this case, properties of the present continuum equations enter on large-scales. 

\medskip
In contrast, in the viscous case $d_1=d_2=0$, $b=b_1=b_2<0$ an analogous scaling analysis shows that the largest real parts of the spectrum for large $|\k|$ approach the negative value $g H_0/b$ as $|\k|\to\infty$. 
In the inviscid case, $d_1=d_2=b_1=b_2=0$ the spectrum for large $|\k|$ approaches the negative value $-\lb/H_0$, and a complex conjugate pair $-\lb/(2H_0) \pm \rmi\sqrt{gH_0}$. Hence, in these cases real parts for large $|\k|$ are bounded uniformly away from zero. 
See \citep{PRY2023} for details.

\medskip
Returning to \cref{e:disp}, by the Routh-Hurwitz criterion, a solution $\lambda$ for fixed $\k$ satisfies $\Re(\lambda)<0$ if and only if 
\begin{equation}\label{e:RH}
a_1>0,\; a_3>0,\; a_1a_2-a_3>0,
\end{equation} 
a zero root $\lambda=0$ if and only if $a_3=0$, and has purely imaginary roots $\lambda=\pm\rmi\omega\in\rmi\R\setminus\{0\}$ if and only if $a_1a_2-a_3=0$ and $a_2>0$. 
Without backscatter $b_j,d_j=0,j=1,2$, the conditions in \cref{e:RH} are satisfied for all wave vectors $\k\neq0$ and any  bottom drag coefficient $\lb>0$, which implies that $(\vh,\eta)=(0,0)$ is spectrally stable in this case (with neutral mode at $\k=0$).

\begin{figure}[t!]
\centering
\subfigure[]{\includegraphics[trim = 4cm 8.5cm 5cm 9cm, clip, height=4.5cm]{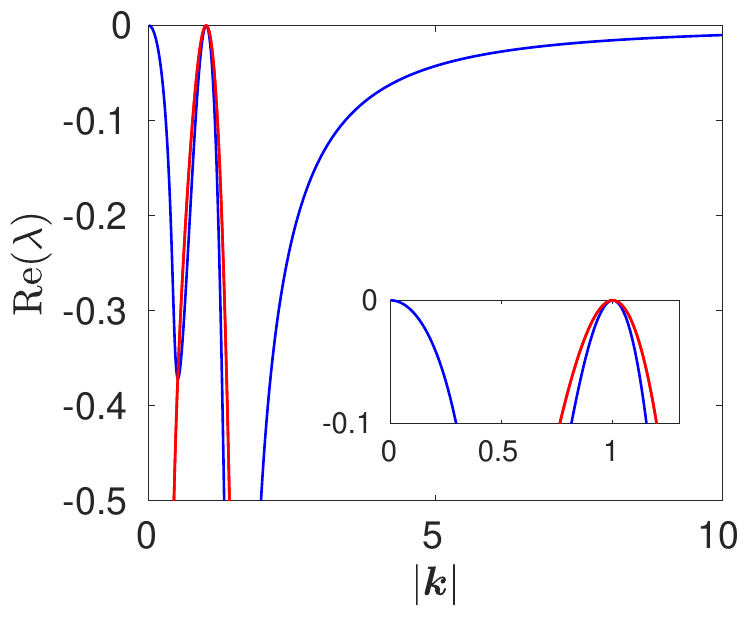}}
\hfil
\subfigure[]{\includegraphics[trim = 4cm 8.5cm 5cm 9cm, clip, height=4.5cm]{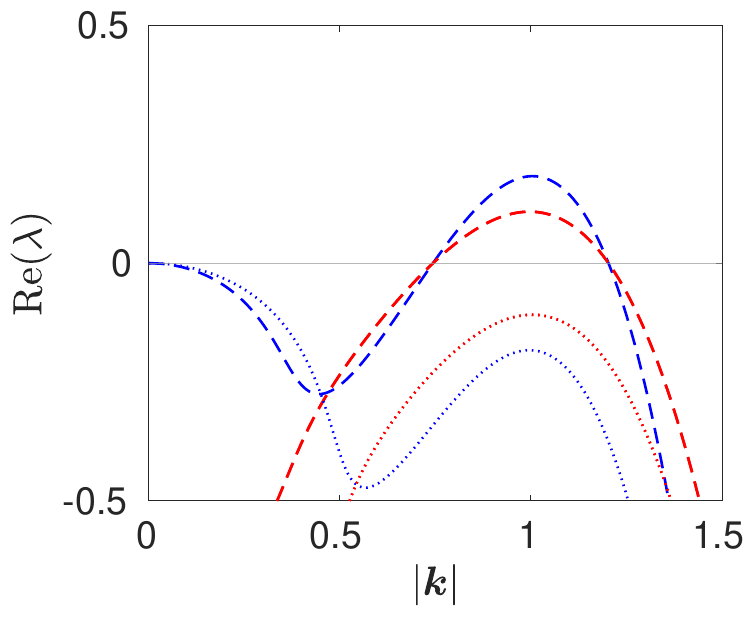}}
\caption{(Fig.~1 in \citep{PRY2023}.) Samples of spectra of $\calL$ from \eqref{e:linop} in the isotropic case $d_1=d_2 =1$, $b_1=b_2 =2$ and $f=0.3$, $g=9.8$, $H_0=0.1$ so that $\lb_\crit=0.1$ and $k_\crit=1$. (a) The real (blue) and complex (red) spectrum (magnification in inset) are simultaneously critical for $\lb=\lb_\crit$ at $|\k|= k_\crit =1$; the real spectrum is also zero at $\k=0$ and approaches zero as $|\k|\to \infty$. (b) Stable case at $\lb = 0.12$ (dotted) and an unstable case at $\lb = 0.08$ (dashed).}
\label{f:specinf}
\end{figure}

For $\lb=0$ and isotropic backscatter $d=d_1=d_2$, $b=b_1=b_2$ the coefficients $a_1$ and $a_3$ equal $dk^4-bk^2$ up to a positive factor. It is then straightforward that stability is precisely determined as in \eqref{e:dispBack}, which is a signature of the fact that solutions to 2D Euler are solutions to the shallow water equations with $\eta=0$. 
More generally, for any $\lb$, in the isotropic case \cref{e:disp} depends on $\k$ only through $K:=|\k|^2$. so that the spectrum is rotationally symmetric in $\k\in\R^2$. Moreover, $a_1$, $a_3$, and thus $a_1a_2-a_3$, possess the factor $F(K;C) := d K^2 - b K + C/H_0$. In addition, $a_1$ and $a_1a_2-a_3$ have the sign of $F$, and $a_3$ that of $KF$. 
It follows that $\lambda=0$ is a solution at $K=0$ and at roots of $F$. The latter emerge when decreasing $C$ below the threshold $C_\crit$ at wave vectors $k_\crit$, where
\begin{equation}\label{e:isocritCk}
\lb_\crit := \frac{b^2H_0}{4d}, \quad |\k|=k_\crit := \sqrt{\frac{b}{2d}}.
\end{equation}
At $C=\lb_\crit$ we denote the solutions by $\lambda_1(k_\crit;\lb_\crit)=0$, $\lambda_2(k_\crit;\lb_\crit)=-\lambda_3(k_\crit;\lb_\crit)=-\rmi \omega_\crit$, where 
\begin{equation}\label{e:omc}
\omega_\crit := \sqrt{a_2}|_{\lb=\lb_\crit, |\k|=k_\crit}=\sqrt{g H_0 k_\crit^2+f^2}.
\end{equation}
We note that for weak backscatter, as $d,b\to 0$ the scalings of $\lb_\crit$ and $k_\crit$ differ, so that fixed $k_\crit$ requires $\lb_\crit\to 0$ and fixed $\lb_\crit$ requires $k_\crit\to\infty$.
 
The continuations of the critical roots satisfy $\blue \Re(\lambda_j(|\k|;\lb_\crit))<0$ for $|\k|\in\R\setminus\{0,k_\crit\}$, $j=1,2,3$ so that, in the wave vector plane, the critical spectrum of $\calL_\crit:=\calL|_{\lb=\lb_\crit}$ from \cref{e:linop} away from the origin forms a circle centered at the origin with radius $k_\crit$, cf.\ Figure \ref{f:specinf}. Moreover, as $C$ decreases below $\lb_\crit$, the zero state becomes unstable \emph{simultaneously} via a stationary (akin to a Turing) instability and an oscillatory (akin to a Turing-Hopf) instability with finite wave numbers in presence of a neutrally stable mode with zero wave number. 

In the anisotropic case $b_1\neq b_2$ and/or $d_1\neq d_2$ 
with $b_j,d_j>0,j=1,2$, the structure of the spectrum is more complicated. Roughly speaking, the real spectrum is always more unstable than any non-real spectrum. In particular, the primary instability with respect to decreasing $\lb$ is purely stationary. We refer the reader to \citep{PRY2023} for a detailed discussion.

\medskip
We can relate critical eigenmodes to those of the inviscid case without drag, and thus the usual geophysical eigenmodes with their terminology, cf.\ e.g.\  \citep{pedlosky1987geophysical,Zeitlin,Vallis2017}.  
Due to the frequency relation \cref{e:omc}, we refer to the corresponding oscillatory eigenmodes as gravity waves (GWs); in the shallow water context these are also called Poincar\'e waves \citep{Vallis2017}. Since the steady modes are in geostrophic balance, we interpret the instabilities as backscatter and bottom drag induced instabilities of selected geostrophic equilibria and gravity waves. This interpretation is also according to the form of the eigenmodes in \S\ref{s:Turing}.

\subsubsection{Notes on global well-posedness for divergence free models}\label{s:wellposed} 

The incompressible models with backscatter discussed in \S\ref{s:PDE}, i.e., the 2D Euler, the 3D primitive and Boussinesq equations, as well as others, such as the quasi-geostrophic potential vorticity equations with backscatter, can be shown to be globally well-posed. This means that these models possess unique global strong solutions, and notably this also holds for bounded spatially and temporally varying backscatter coefficient $b$. These results will be presented in detail in an upcoming publication and here we outline the setting and relation to the literature. 

For a typical definition of a strong solution, the 2D Euler equations with backscatter serve as a representative example, where it is defined as follows: Let $\u_0 \in L_{2,\div}(\T^2;\R^2)$. Then, a function $\u \in L_{2,\textnormal{loc}}(0,\infty;L_{2,\div}(\T^2;\R^2))$ is called a \textit{global strong solution} of the Euler equations with backscatter if $\u \in L_{2,\textnormal{loc}}(0,\infty;H_{\div}^4(\T^2;\R^2)) \cap H_{\textnormal{loc}}^1(0,\infty;L_{2,\div}(\T^2;\R^2))$, $\u(0) = \u_0$, and \eqref{eq:2DNSa} holds with $\u$ for a.e. $t \in (0,\infty)$; here
\begin{align*}
L_{2,\div}(\T^2;\R^2)&:= \bigg\{\w \in L_2(\T^2;\R^2); \ \int_{\T^2} \w \cdot \nabla \boldsymbol{\varphi} \, \dif\x = 0 \ (\boldsymbol{\varphi} \in H_0^1(\T^2;\R^2)) \bigg\}, \\
H_{\div}^k(\T^2;\R^2) &:= H^k(\T^2;\R^2) \cap L_{2,\div}(\T^2;\R^2) \qquad (k \in \N). 
\end{align*}
The existence of unique global strong solutions, i.e., global well-posedness, can typically be shown for the aforementioned models under the condition that the initial data satisfy $H^2$-regularity. For instance, in the case of the 2D Euler equations, this condition is met when $\u_0 \in H_{\div}^2(\T^2;\R^2)$.

While the backscatter parameterization injects energy into the underlying system at the large scales,  the small scales are strongly damped by the fourth-order hyperviscosity term in the backscatter parameterization. 
As will be discussed in detail elsewhere, the apparent lack of control of large scales is compensated by the transport nonlinearity $(\u \cdot \nabla)\u$, in the sense that it 
transfers energy from the large scales to the strongly controlled small scales with sufficient strength 
to ensure global well-posedness. 

When replacing the backscatter parametrization with a viscosity term $\nu \Delta$ in the aforementioned models we obtain the conventional models, such as the 2D Navier-Stokes equations instead of the Euler equations with backscatter. For these cases, except the 3D Boussinesq equations, global well-posedness is well known; 
see, for example, \citep[Prop. 3.2]{M08}, \citep{CT07}, \citep[\S 3.2]{T83} and the references therein. The standard method used in these references to show global well-posedness involves proving energy estimates and the use of the Galerkin truncations, cf. \citep[\S 30]{Zeidler90}. 
This can also be applied to the aforementioned models with backscatter, where the additional regularity due to the hyperviscous term typically simplifies some aspects. 
For instance, strong solutions of the 2D Euler equations with backscatter, as $H^4$-functions, possess higher regularity than strong solutions of the 2D Navier-Stokes equations, which are $H^2$-functions. 
Consequently, stronger Sobolev embedding theorems can be applied which directly yield the desired energy estimates. 
In the proofs of global well-posedness for the aformentioned models without backscatter, it may be necessary to establish  
additional energy estimates, so that proofs require more steps compared to the case with backscatter. 

It is worth noting that this approach works well with divergence-free vector fields condition, 
as the nonlinearity vanishes at certain points within the calculations. 
For models that admit divergence of the velocity vector, 
such as the rotating shallow water equations with backscatter or the Kuramoto-Sivashinsky equations, using this method to establish global well-posedness is generally less promising. 
However, in general one can show the existence of mild solutions, at least locally in time, and, e.g., global classical solutions for the dissipative shallow water equations without backscatter \citep{K85}. Interestingly, as discussed in \S\ref{s:spec} for shallow water with the chosen backscatter, the spectrum does not yield a spectral gap. This suggests there is less smoothing than for the viscous case, but we do not go into further details here.

\medskip
Having established global existence, we can ask about the long term dynamics. 
From the discussion of the spectra of the linearizations it is clear that the ratio of $d$ and $b$ plays an important role, in particular on bounded domains. 
Indeed, for $b<d$ and $\Omega_h=\T^2$ it is straightforward for the 2D Euler equations \eqref{eq:2DNS} with zero mean conditions (i.e., $\int_{\T^2} \u\,\dif\x = 0$) that $\u=0$ is the global attractor: 
Recall the energy equation \eqref{e:2DNS_en}. 
From $\int_{\T^2}\nabla \u \,\dif\x = 0$ and the Poincar\'e inequality $\|w\|_2 \leq \|\nabla w\|_2$ for $w \in H^1(\T^2)$ with $\int_{\T^2} w \,\dif\x = 0$ it follows for $b < d$ that $d\|\Delta \u\|_2^2 - b\|\nabla \u\|_2^2 \geq (d-b)\| \u\|_2^2$, and thus 
\[
\frac{1}{2}\frac{\dif}{\dif t} \|\u\|_2^2 + (d-b)\|\u\|_2^2 \leq 0.
\]
Applying Gr\"onwall's inequality, we obtain $\|\u\|_2 \leq \rme^{-(d-b)t}$, which is precisely the exponential decay predicted from the spectrum of the linearization. For the Boussinesq equations \eqref{eq: introB} with constant $N^2>0$ the situation is similar, with positive decay rate for $b<d$ that also depends on 
$N^{2},\mu$ and $\nu_v$; analogously for the primitive equations \eqref{eq: primEq}.

For $b>d$ the linearization of the 2D Euler with backscatter \eqref{eq:2DNS} predicts instability, which precludes global convergence to zero. Since the unstable eigenspace is four-dimensional, an attractor has at least that dimension. However, as discussion in the next section, certain solutions in the unstable eigenspace are independent of the transport nonlinearity and thus grow exponentially and unboundedly. Hence, there does not exist a classical global compact attractor. In fact, the unbounded growth occurs in open sets of phase space as shown in \S\ref{s:stabgrow}. Since the 2D Euler equations are an invariant subsystem in the Boussinesq and primitive equations with simple boundary conditions, such solutions also occur for these.

\section{Plane waves, superpositions and unbounded growth}\label{s:PlaWave}

In this section, we seek insight into the impact of backscatter with constant coefficients through certain specific waves and flows. We start from explicit plane wave type solutions that depend on a single phase variable, that couples spatial and temporal dependence, and for which the nonlinear terms vanish or is a gradient that only impacts the pressure. 

The rotating 2D Euler equations \eqref{eq:2DNS} admit the one-dimensional space of solutions
\begin{align}\label{e:2DNSlargescale}
\vh(t,x,y)= \psi(t,y)\begin{pmatrix}1\\0\end{pmatrix}, \; \psi(t,y):= A\rme^{(b - d)t}\cos(y), \; A\in\R.
\end{align}
For $f=0$ this is evident with $p=0$, and for $f\neq 0$ the associated pressure is $p=-fA\rme^{(b - d)t}\sin(y)$.
Notably, \eqref{e:2DNSlargescale} is a solution for arbitrary $A$ since it lies in the `kernel' of the transport nonlinearity for $f=0$, or is mapped to a gradient for $f\neq 0$ by it. 

In particular, these solutions are steady for $b=d$, decay exponentially for $b<d$ and grow exponentially and unboundedly for $b>d$. Such unbounded growth can occur in all models introduced in \S\ref{s:PDE} and is a remarkble explicit signature of the energy input through backscatter. In smooth dynamical systems such unbounded growth is sometimes called `grow-up', in analogy to the finite-time `blow-up' \citep{BenGal}. 
We show in \S\ref{s:stabgrow}  that for \eqref{eq:2DNS} this is even a dynamically stable phenomenon. 

By symmetry, translations of the solutions \eqref{e:2DNSlargescale} are also solutions, so that cosine can be replaced by sine,  and we can also jointly replace $y$ by $x$ and the vector $(1,0)$ by $(0,1)$. 
All these are also  real eigenmodes of \eqref{e:Lin2DNS}, but their superposition does not necessarily yield a solution to the nonlinear \eqref{eq:2DNS}. However, we can superpose different spatial translates of \eqref{e:2DNSlargescale} and the $x$-dependent counterpart, which gives two 2D spaces of solutions (instead of the full 4D eigenspace). Concerning the wave numbers,  \eqref{e:2DNSlargescale} is `large scale' on the unit torus $\T^2$, but solutions with higher wave numbers exist by replacing $\cos(y)$ with $\cos(k y)$ for any $k\in\N$ and adjusting $b,d$ in the exponential. When posed on $\R^2$ we can consider any wave length, similar to the next subsection. See also \citep{PR2020} and the references therein. All such solutions also called monochromatic since they depend on a single wave number. The explicit solutions we discuss rely on some orthogonality relation of wave vector and velocity. With some abuse of terminology, due to this plane wave structure, we sometimes refer to such solutions as waves or flows, although this section  is not intended for a broader study of waves. 

\medskip
Before we consider the rotating shallow water and Boussinesq equations in more detail, we note that the primitive equations admit such solutions as well. More precisely, it is straightforward to see that the primitive equations \eqref{eq: primEq} with $f = 0$ subject to the boundary conditions \eqref{eq: primEqBC} admit the solutions
\[
 T(t,\x):= 0, \; \u(t,\x):= A\rme^{(b - d - \nu_{v}\omega^2)t} \cos(\omega z) \cos(ky)\begin{pmatrix}
   1 \\ 0
\end{pmatrix}, \; A \in \R,
\]
where $\omega \in (-\pi/2,\pi/2)$ is chosen such that $\beta = \omega \tan(\omega H)$. The latter is equivalent to the explicit solutions $\u$ satisfying the bottom drag boundary condition $\frac{\partial \u}{\partial z}|_{z = -H} = \beta \u|_{z = -H}$ from equation \eqref{eq: primEqBC}. In particular, these solutions are steady for $b=d + \nu_v \omega^2$, decay exponentially for $b<d + \nu_v \omega^2$ and grow exponentially and unboundedly for $b>d + \nu_v \omega^2$.

\subsection{Rotating shallow water}\label{s:swe}
For \eqref{e:sw} with $\F=0$ the analogous explicit solutions result from the plane wave ansatz 
\begin{align}\label{sol: RSWB1}
\vh = \psi(t,\k\cdot\x)\k^{\perp}\,, \quad \eta = \phi(\k\cdot\x)\,,
\end{align}
with phase variable $\xi=\k\cdot\x$. In particular, the nonlinear terms vanish due to the orthogonality of $\k$ and $\k^\perp$. On the one hand, this implies time-independence of $\eta$ due to \eqref{e:swb} since $\vh$ is divergence free. On the other hand, inserting \eqref{sol: RSWB1} into \eqref{e:swa} for $\F=0$ 
and taking scalar products with non-zero $\k^\perp$ and $\k$, respectively, result in the linear equations
\begin{subequations}\label{cond: RSWB1}
\begin{align}
\frac{\p\psi}{\p t} &= -|\k|^2\left(d_1\ky^2+d_2\kx^2\right)\frac{\p^4\psi}{\p\xi^4}-\left(b_1\ky^2+b_2\kx^2\right)\frac{\p^2\psi}{\p\xi^2}\label{cond: RSWB1a} \\%[2mm]
g\frac{\p\phi}{\p\xi}-f\psi &= \kx \ky\left((d_1-d_2)|\k|^2\frac{\p^4\psi}{\p\xi^4}+(b_1-b_2)\frac{\p^2\psi}{\p\xi^2}\right)\,.\label{cond: RSWB1b}
\end{align}
\end{subequations}
For isotropic backscatter $b_1=b_2$, $d_1=d_2$, \eqref{cond: RSWB1b} is a geostrophic balance relation of velocity and pressure. 
Otherwise, in \eqref{cond: RSWB1b} the pressure gradient compensates a mix of Coriolis and backscatter terms, thus the velocity field \eqref{sol: RSWB1} with \eqref{cond: RSWB1} is in general not geostrophically balanced.

For monochromatic solutions that contain a single non-trivial Fourier mode, \eqref{cond: RSWB1} implies the form
\begin{align}\label{sol: RSWB2}
\vh = \alpha_1\rme^{\lambda t}\cos(\k\cdot\x + \tau)\k^{\perp}\,, \quad \eta = \alpha_2\frac{f}{g}\sin(\k\cdot\x + \tau) + s\,,
\end{align}
with arbitrary shifts $\tau,\, s\in\R$ and the other real parameters must satisfy the algebraic equations
\begin{subequations}\label{cond: RSWB2}
\begin{align}
\lambda &=  (b_1-d_1|\k|^2)\ky^2+(b_2-d_2|\k|^2)\kx^2 \label{cond: RSWB2a} \\
\frac{\alpha_2-\alpha_1}{\alpha_1}f &= \kx \ky \bigl((d_1-d_2)|\k|^2+b_2-b_1\bigr)\label{cond: RSWB2b}\\
\alpha_2\cdot\lambda&=0\label{cond: RSWB2c}\,.
\end{align}
\end{subequations}
Here \eqref{cond: RSWB2a} is a dispersion relation of growth/decay and wave vector, and \eqref{cond: RSWB2b} is an amplitude relation for $\alpha_1$ and $\alpha_2$ in terms of the wave vector so that solutions come at least in one-dimensional spaces with constant $(\alpha_2-\alpha_1)/\alpha_2$. The compatibility condition \eqref{cond: RSWB2} means $\alpha_2=0$, i.e., constant $\eta$, or $\lambda=0$, i.e. a steady solution. In particular, these explicit solutions with non-trivial depth variation $\eta$ are steady. Solutions with $\lambda> 0$ grow exponentially and unboundedly due to arbitrary amplitude {\blue $\alpha_1$}. In Figure~\ref{Fig. 1b} we plot the loci of these different solutions in the wave-vector plane. Certain multi-chromatic solutions can be constructed by superposition of solutions on the same ray in the wave-vector plane. See \citep{PRY22} for details. This \textit{radial superposition principle} in the example of Figure \ref{Fig. 1c} gives three-dimensional subspaces in which the dynamics are linear. Clearly, all this is strongly constrained when allowing only for isotropic backscatter.

\begin{figure}
\begin{center}
\subfigure[]{
\includegraphics[trim=5cm 8.5cm 5.5cm 8.5cm, clip, height=5cm]{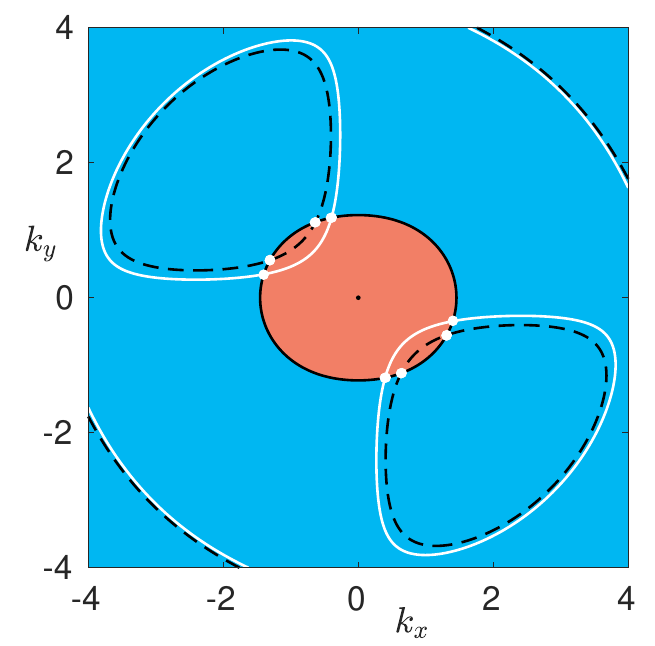}\label{Fig. 1b}}
\hfil
\subfigure[]{
\includegraphics[trim=5cm 8.5cm 5.5cm 8.5cm, clip, height=5cm]{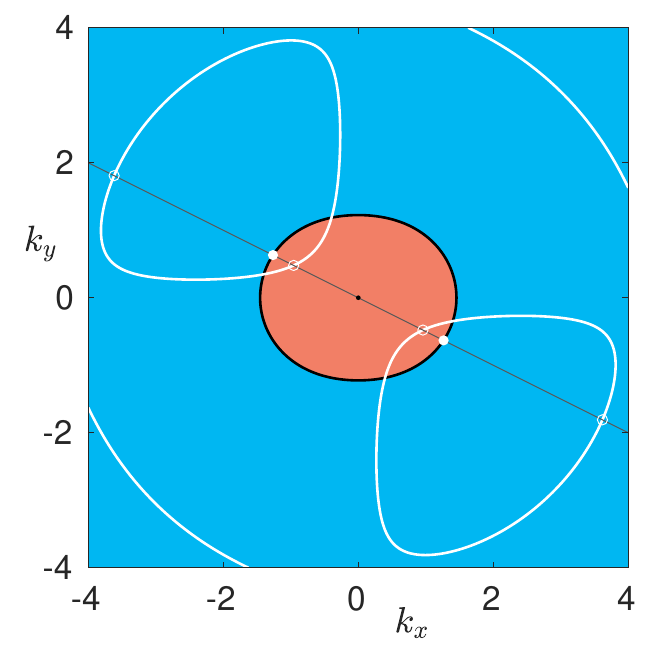}\label{Fig. 1c}}
\hfil
\subfigure[]{
\includegraphics[trim=5cm 8.5cm 5.5cm 8.5cm, clip, height=5cm]{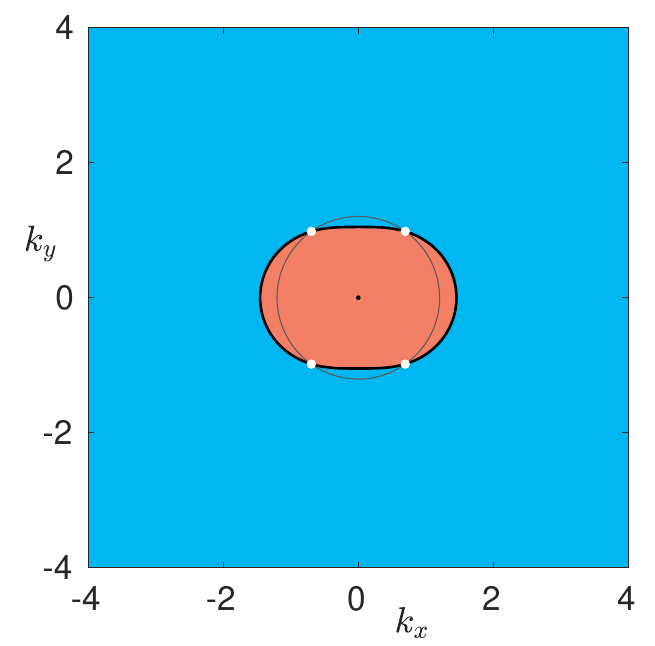}
\label{Fig. RadSuperpositionb}}
\caption{\label{Fig. 1} 
(Taken from Figs.~1, 8 in \citep{PRY22}.)
We plot part of the wave vector plane where an associated growth rate is $\lambda>0$ (red), $\lambda=0$ (black), $\lambda<0$ (blue). 
Panels (a,b) concern the loci of explicit solutions \eqref{sol: RSWB2} with fixed parameters $d_1=1.0,\, d_2=1.04,\, b_1=1.5,\, b_2=2.2,\, f=0.3,\, g = 9.8,\, H_0 = 0.1,\, \alpha_1=1.0$ and $\alpha_2=-0.5$. White curves: loci of solutions with $\alpha_2=0$; white bullets: loci of steady solutions; black dashed: loci of solutions satisfying \eqref{cond: RSWB2b} only. Panel (c) concerns \eqref{cond: B-intsupa} with  $d_1, d_2,b_2$ as in (a,b) and $b_1=1.1$ so white bullets mark the loci of steady solutions. In (b,c) the solutions corresponding to intersections with an arbitrary gray line in (b) or arbitrary circle in (c), marked by white circles and bullets, can be linearly superposed and still yield explicit solutions. 
}
\end{center}
\end{figure}

\subsection{Boussinesq equations}
The Boussinesq equations \eqref{eq: introB} allow for many more explicit solutions, in particular on a vertically periodic domain. Horizontal flows are a class of these that is closely related to those of \eqref{e:sw}. These are barotropic with horizontal velocity field (i.e., $\v$ is independent of the vertical coordinate $z$), vanishing vertical velocity $w\equiv 0$, and horizontally constant buoyancy $\b=\b(t,z)$. Inserting this ansatz into \eqref{eq: introB} yields $p=\tilde{p}+\mathtt B(t,z)$, where $\frac{\p \mathtt B(t,z)}{\p z}=\b(t,z)$, and the decoupled system
\begin{subequations}\label{eq: rB}
\begin{align}
\frac{\p\vh}{\p t}+(\vh\cdot\nabla)\vh+f\vh^{\perp}+\nabla\tilde{p} &= - \mathrm{diag}\bigl(d_1\Delta+b_1,d_2\Delta+b_2\bigr)\Delta\vh\label{eq: rBa}\\
\nabla\cdot\vh&= 0\label{eq: rBb}\\
\frac{\p \b}{\p t} &= \mu\frac{\p^2}{\p z^2}\b\,,\label{eq: rBc}
\end{align}
\end{subequations}
with gradient and Laplace operators for the horizontal directions $\x=(x,y)$,  $\vh=\vh(t,\x)\in\R^2$. In the simplest case, the heat equation for the buoyancy  \eqref{eq: rBc} can be readily solved by Fourier transform. 

\medskip
The momentum equations are, for these solutions, less restrictive than the shallow water equations and admit a larger set of explicit flow solutions; see also \citep{PR2020} for the setting without backscatter. Explicit solutions of \eqref{eq: rBa} and \eqref{eq: rBb} for which the nonlinear terms vanish can be identified by the ansatz for wave shape $\psi$ and pressure profile $\phi$
\begin{align}\label{sol: B}
\vh = \psi(t,\k\cdot\x)\k^{\perp}\,,\quad \tilde{p}= g\phi(t,\k\cdot\x)\,,
\end{align}
where $\k\in\R^2$ and without loss of generality $|\k|=1$ by the freedom in choosing $\psi$ and $\phi$. Equation \eqref{eq: rBb} is always satisfied and inserting \eqref{sol: B} into \eqref{eq: rBa} gives \eqref{cond: RSWB1}, a signature of the relation between horizontal flow and the shallow water equations. 

The general wave shape $\psi$ in \eqref{sol: B} contains superposition of arbitrarily many monochromatic waves in the \textit{same wave vector direction} $\k$ and, in contrast to \eqref{e:sw}, \textit{any wave number} $|\k|$. The  Boussinesq equations also admit an \textit{angular superposition principle} of $\vh$ in the form \eqref{sol: RSWB2} with \textit{different wave vector directions} and the \textit{same wave number}. See an example in Figure~\ref{Fig. RadSuperpositionb}, and we refer the reader to \citep{PR2020} for details. 
By integrating over the whole circle $S_{\sc}:=\{\k\in\R^2\mid |\k|=\sc\}$ for any fixed $\sc>0$ the exact form is then
\begin{subequations}\label{sol: B-intsup}
\begin{align}
\vh(t,\x)=&\ \int\displaylimits_{S_{\sc}}\alpha(\k)\rme^{\lambda(\k)t}\sin\bigl(\k\cdot\x+ \tau(\k)\bigr)\k^{\perp} \dif \k, \label{sol: B-intsupa}\\
\begin{split}
\tilde{p}(t,\x)=&\ -\sc^4\int_0^{2\pi}\int_{\varphi_1}^{2\pi} \alpha_1\alpha_2\,\rme^{(\lambda_1+\lambda_2)t}\big(\cos\xi_1\cos\xi_2+\cos(\varphi_1-\varphi_2)\sin\xi_1\sin\xi_2\big)\,\dif\varphi_2\,\dif\varphi_1\\
&\ -f\int\displaylimits_{S_{\sc}}\gamma(\k)\rme^{\lambda(\k)t}\cos\bigl(\k\cdot\x+ \tau(\k)\bigr)\,\dif\k, \label{sol: B-intsupb}
\end{split}
\end{align}
\end{subequations}
where, for $i=1,\,2$, we set $\alpha_i:=\alpha(\k_i)$ with $\k_i:=\sc(\cos(\varphi_i),\sin(\varphi_i))$, $\lambda_i:=\lambda(\k_i)$ and $\xi_i:= \k_i\cdot\x+ \tau(\k_i)$, for all $0\leq\varphi_i<2\pi$. 
These must (for almost all $\k\in S_{\sc}$) be taken from the solution set to the dispersion and amplitude relations
\begin{subequations}\label{cond: B-intsup}
\begin{align}
\alpha(\k)\lambda(\k)&=\alpha(\k)\bigl((b_1-d_1\sc^2)\ky^2+(b_2-d_2\sc^2)\kx^2\bigr)\,,\label{cond: B-intsupa}\\
f\bigl(\gamma(\k)-\alpha(\k)\bigr)&=\alpha(\k)\kx \ky\bigl((d_1-d_2)\sc^2+b_2-b_1\bigr)\label{cond: B-intsupb}\,.
\end{align}
\end{subequations}
Sufficient for the convergence of the integrals is $\alpha\in L_1(S_{\sc})$, $\tau\in L_{\infty}(S_{\sc})$.

Notably, for these explicit solutions the nonlinear terms are in general not vanishing, but are a gradient that can be fully compensated by the pressure gradient. See \citep{PR2020} for further discussion and references for explicit solutions with gradient nonlinearities without backscatter.

\subsubsection{Flows with vertical structure and coupled buoyancy}\label{Rotating Boussinesq coupled}
The rotating Boussinesq equations with backscatter \eqref{eq: introB} also admit explicit solutions in which the velocity and the buoyancy are coupled, and in which the vertical dependence and velocity component is non-trivial. These again can be related to geophysical flows and waves in the viscous or inviscid case. Here we briefly showcase some parallel flows, Kolmogorov flows and monochromatic internal gravity waves and refer the reader to \citep{PRY22} for further details.

\paragraph{Parallel flow}\label{s:parflow}
As is well-known in the inviscid and viscous case, e.g. \citep{Wang90}, a purely vertical velocity  
admits solutions of the form 
\begin{align}\label{sol: Parallel}
\v(t,\x)=w(t,x,y)\3 \,,\quad \b(t,\x)=\tb(t,x,y) \,,\quad p(t,\x)=\tilde{p}(t)z \,,
\end{align}
where $w$ and $\tb$ satisfy (with horizontal Laplacian)
\begin{subequations}\label{cond: Parallel}
\begin{align}
\frac{\p w}{\p t} - \nu_v \Delta w+\tilde{p}&=\tb\label{cond: Parallela}\\
\frac{\p\tb}{\p t} - \mu\Delta\tb&=-N^2w\,.\label{cond: Parallelb}
\end{align}
\end{subequations}
In particular, these parallel flows are independent of horizontal kinetic energy backscatter. They can be superposed with  horizontal flows \eqref{sol: B} that have zero buoyancy, if their wave vector directions $\k$ are the same. Therefore, any such parallel flow is unboundedly unstable in \eqref{eq: introB} through perturbations of the form \eqref{sol: B} with the same wave vector $\k$ and small enough wave number $|\k|$. 
 
\paragraph{Kolmogorov flow}
Another well-known class of explicit solutions for the Boussinesq equations, see e.g.\ \citep{BalmforthYoung2005}, take the form
\begin{align}\label{sol: Kolmogorov}
\v(t,\x)= \rme^{\lambda t}\cos(\k\cdot \x)\a \,,\quad \b(t,\x)=c\rme^{\lambda t}\cos(\k\cdot \x) \,,\quad
p(t,\x)=\gamma \rme^{\lambda t}\sin(\k\cdot \x) \,,
\end{align}
with wave vectors of the form $\k=(k,0,-m)$, $k,m\in\R$, 
and the flow direction $\a=\alpha(0,1,0) +\beta(m,0,k).$ In order to solve the Boussinesq equations \eqref{eq: introB}, the coefficients of \eqref{sol: Kolmogorov} have to satisfy
\begin{align}\label{cond: coeff}
\begin{pmatrix}
-f & m(\lambda+\delta_1) & 0 & k\\
\lambda+\delta_2 & fm & 0 & 0\\
0 & k(\lambda+\delta_3) & -1 & -m\\
0 & N^2k & \lambda+\delta_\mu & 0
\end{pmatrix}
\begin{pmatrix}
\alpha \\ \beta \\ c \\ \gamma
\end{pmatrix}
=0\,,
\end{align}
where $\delta_{\mu}:=\mu|\k|^2$, $\delta_i:=d_i|\k|^4-b_i|\k|^2$ for any $1\leq i\leq3$. The non-trivial explicit solutions require vanishing determinant of this matrix, and its roots $\lambda=\lambda(\k)$ determine the linear dynamics of \eqref{sol: Kolmogorov}. 
Superposition of different Kolmogorov flows is possible, as long as all wave vectors $\k$ have the same direction, which can prove unbounded instability of steady Kolmogorov flows. 
{\blue For instance, for stable stratification $N^2>0$, the simplest case is zero thermal diffusivity $\mu=0$, where Kolmogorow flows with wavelength $|\k|^2<b_2/d_2$ grow unboundedly and those with $|\k|^2=b_2/d_2$ are steady and unboundedly unstable against perturbations with growing Kolmogorov flows.}%\JYx{Slight edits, ok?}\JRx{Ok for me}

\paragraph{Internal gravity waves (IGW)}
This kind of explicit (and also monochromatic) solution structurally differs from the aforementioned flows and is prominent in geophysics, cf.\ e.g.\ \citep{Achatz06}. The wave profile of an IGW is a time-dependent travelling wave with phase variable $\xi=kx+mz-\omega t$ and takes the form
\begin{align*}
\v(t,\x)&= \alpha_1 \rme^{\lambda t}\sin(\xi)(0,1,0) + \alpha_2\omega \rme^{\lambda t}\cos(\xi)(-m,0,k) \,,\\
\b(t,\x)&= \beta_1 \rme^{\lambda t}\sin(\xi) + \beta_2\omega \rme^{\lambda t}\cos(\xi) \,,\\
p(t,\x)&= \gamma_1 \rme^{\lambda t}\cos(\xi) + \gamma_2\omega \rme^{\lambda t}\sin(\xi)\,.
\end{align*}
The conditions for these to be (non-trivial) explicit solutions of \eqref{eq: introB} in particular depend on $\omega$. In fact, for $\omega=0$ these are Kolmogorov flows from \eqref{sol: Kolmogorov} with $\beta=0$.
For $\omega\neq0$ the existence conditions are given by eight equations that are linear in $\alpha_j,\beta_j,\gamma_j$, and we refer the reader to \citep{PRY22} for details. 
{\blue A simple example are IGW with $\omega= \pm f$, $k=0$, $\alpha_1=-\alpha_2 m f$, $\beta_1=\beta_2=\gamma_1=\gamma_2=0$, $d_1=d_2, b_1=b_2$, 
whose growth rate is $\lambda = (b_2-d_2m^2)m^2$.} 
While Kolmogorov flows exist on the whole wave vector space, IGW in general do not. 
It is possible to superpose IGWs and Kolmogorov flows as a solution to \eqref{eq: introB} if these have the same direction of wave vectors, and this can also be in the form of an integral. 
With this superposition one can prove in several cases the unbounded instability of steady IGWs due to perturbations with exponentially growing Kolmogorov flows, and vice versa. 
{\blue The mentioned examples provide simple cases for this, and again we refer to \citep{PRY22} for more details.} %\JRx{Maybe we can cite some small analytical result? At least I would write explicitly that we have such results.}

\medskip
The exhibition of samples in this section illustrates that unbounded growth for constant backscatter coefficients occurs in a variety of models and flows. This highlights that such backscatter can be too strong and lead to accumulation of energy into selected modes. The recipes given can also be applied to other models in particular the rather simple quasi-geostrophic equations, but also the primitive equations \eqref{eq: primEq} and certain more physical boundary conditions. We also note that since these solutions are linear modes, they also exist for Galerkin-truncations of the models, which is sometimes used instead of finite element, volume or difference discretizations.

\section{Stability of unbounded growth}\label{s:stabgrow}

In this section we show that the unbounded growth observed for the 2D Euler equations with backscatter by \eqref{e:2DNSlargescale} is even dynamically stable, as formulated in Theorem~\ref{t:stab2DNSE}. Although technically different, we note that such stability at infinity has conceptual similarity with stability of finite-time blow-up, e.g.\ \citep{GNT2010,RS2013}. 
Stable unbounded growth in this case holds for any non-zero initial value $\w(0)=\w_0$ lying in the $4$-dimensional vector space spanned by the initial values of \eqref{e:2DNSlargescale} and its translations; 
recall that \eqref{e:2DNSlargescale} and its translations produce solutions of \eqref{eq:2DNS} exhibiting unbounded and exponential growth when $b > d$. We also note that all these have mean zero.

For technical reasons, in this section we consider \eqref{e:2DNSlargescale} on $\T^2= \R^2/ 2\pi \Z^2$  restricted to the invariant subspace with zero mean, i.e., the system
\begin{subequations}\label{eq:2DNS-2}
\begin{align}
\p_t \u+ (d\Delta^2 + b\Delta)\u + (\u\cdot\nabla)\u + f\u^{\perp} + \nabla p &=  0,\label{eq:2DNS-2a}\\
\divh \u &= 0,\label{eq:2DNS-2b} \\
\int_{\T^2} \u \,\dif\x &= 0.
\end{align}
\end{subequations}
The assumption of zero mean has no impact on the stability of unbounded growth, since this is inherited from the mean zero projection as follows. We first observe that the mean $\boldsymbol{m}:=(2\pi)^{-2}\int_{\T^2}\u\,\dif\x\in\R^2$ for a solution $\u$ to \eqref{eq:2DNS} satisfies $\partial_t \boldsymbol{m} +f\boldsymbol{m}^\perp=0$ so that the length of $\boldsymbol{m}(t)\in\R^2$ remains constant (and $\boldsymbol{m}$ oscillates, unless $f=0$). The projection onto mean zero $\tilde \u:=\u-\boldsymbol{m}$ satisfies \eqref{eq:2DNS-2} with the additional drift term $\boldsymbol{m}\cdot\nabla\u$ on the left hand side of \eqref{eq:2DNS-2a}. This term is removed in the natural moving frame, i.e., $\check \u(t,\x):= \tilde \u\big(t,\x +\int_0^t \boldsymbol{m}(s){\rm d} s\big)$ satisfies \eqref{eq:2DNS-2}. Since $\|\check \u(t)\|_2= \|\tilde \u(t)\|_2$, and $\boldsymbol{m}$ has constant length, unbounded growth of $\u$ solving \eqref{eq:2DNS} with arbitrary mean is inherited from that of $\check \u$ solving \eqref{eq:2DNS-2}, which we discuss in this section.

\medskip
We remark that the rotating shallow water equations \eqref{e:sw} without drag $\F=0$ on $\T^2$, and restricted to the invariant subspace $\eta = 0$,  $\int_{\T^2}\vh\,\dif\x=0$, coincide with \eqref{eq:2DNS-2}. Consequently, the results of this section regarding the dynamics of system \eqref{eq:2DNS-2} equally hold for this special case of \eqref{e:sw}; again the assumption of zero mean is not relevant. 
Similarly, the 2D rotating Euler equations are an invariant subspace in the Boussinesq equations \eqref{eq: introB} for $\b=0$ and horizontal $\v$ that is independent of $x_3$; analogously for the primitive equations \eqref{eq: primEq} without forcing and {\blue $\beta = \gamma = 0$}, $\boldsymbol{\tau}=0$, $T^*=0$.

\medskip
Let us define the spaces of divergence-free vector fields with zero mean in $H^k$ and $L_2$,
\begin{align*}
H_{\div}^{0,k}(\T^2;\R^2)&:= \bigg\{\u \in H^k(\T^2;\R^2); \ \divh \u = 0, \ \int_{\T^2} \u \,\dif\x = 0 \bigg\} \qquad (k \in \N), \\  L_{2,\div}^0(\T^2;\R^2)&:= \bigg\{\u \in 
L_2(\T^2;\R^2); \ \int_{\T^2} \u \cdot \nabla \varphi\,\dif\x = 0 \ (\varphi \in H_0^1(\T^2;\R^2)), \ \int_{\T^2} \u \,\dif\x = 0 \bigg\}.
\end{align*}
We further denote by $P$ the Leray projection, which is the orthogonal projection of $L_{2,\div}^0(\T^2;\R^2)$ along $\{\u \in L_2(\T^2;\R^2); \ \smallint_{\T^2} \u \,\dif\x = 0 \}$, and define $B(\u,\v):= P((\u \cdot \nabla)\v)$ as a simplified expression of the advection term in \eqref{eq:2DNS-2}. Last but not least, we denote by $A$ the Stokes operator on $\T^2$, which is an unbounded operator in $L_{2,\div}^0(\T^2;\R^2)$ and defined by
\[
\operatorname{dom}(A):= H_{\div}^{0,2}(\T^2;\R^2), \quad A\u := -\Delta \u.
\]
Now we can rewrite \eqref{eq:2DNS-2} as the following evolution equation:   
\begin{subequations}\label{eq:2DNS-3}
\begin{align}
\frac{\dif}{\dif t} \u + d A^2\u - b A\u + B(\u,\u) + fP(\u^{\perp}) &= 0, \label{eq:2DNS-3a} \\ 
\u(0) &= \u_0; \label{eq:2DNS-3b}
\end{align}
\end{subequations}
note that the operator $-(dA^2-bA) - fP\Id^{\perp}$ relates to the operator $L$ considered in \eqref{e:Lin2DNS}. As indicated in \S\ref{s:wellposed}, the evolution equation \eqref{eq:2DNS-3} is globally well-posed. To wit, for all $\u_0 \in H_{\div}^{0,2}(\T^2;\R^2)$ there exists a 
\[
S(\cdot)\u_0:= \u \in L_{2,\textnormal{loc}}(0,\infty;H_{\div}^{0,4}(\T^2;\R^2)) \cap C([0,\infty);H_{\div}^{0,2}(\T^2;\R^2))
\] 
such that $\u(0) = \u_0$ and \eqref{eq:2DNS-3} holds in $L_2(\T^2;\R^2)$ for a.e. $t \in (0,\infty)$. Note that $S:[0,\infty) \to M(H_{\div}^{0,2}(\T^2;\R^2))$ defines a (nonlinear) semigroup, where we denote by $M(H_{\div}^{0,2}(\T^2;\R^2))$ the set of all maps from $H_{\div}^{0,2}(\T^2;\R^2)$ into itself.

It is well-known that there exists an orthonormalbasis $(\e_n)_{n \in \N}$ of 
$L_{2,\div}^0(\T^2;\R^2)$ consisting of eigenvectors of $A$, with an increasing sequence of eigenvalues $(\lambda_n)_{n \in \N}$ in $(0,\infty)$, such that 
\begin{align}
\e_1(x,y)= \begin{pmatrix}\cos(y)\\0\end{pmatrix}, \quad \e_2(x,y) =  \begin{pmatrix}\sin(y)\\0\end{pmatrix}, \quad \e_3(x,y) = \begin{pmatrix}0\\ \cos(x)\end{pmatrix}, \quad \e_4(x,y) = \begin{pmatrix}0\\ \sin(x)\end{pmatrix} \label{eq:eigfunct}  
\end{align}
and $\lambda_1=\lambda_2=\lambda_3= \lambda_4 = 1 < 2 =\lambda_5$.

As discussed in \S\ref{s:wellposed}, it turns out that for small backscatter coefficient $b$, more precisely $b < d= d\lambda_1$ in case of the evolution equation \eqref{eq:2DNS-3}, all global strong solutions of \eqref{eq:2DNS-3} decay exponentially to $0$ for $t \to \infty$, so that the global attractor of $S$ is $\u=0$. However, 
we recall from the beginning of \S\ref{s:PlaWave} that for $b > d$
\begin{equation}\label{eq:explsol}
S(t)(\e_j) = \rme^{-(d - b) t} \e_j  \qquad (t \geq 0)
\end{equation}
is an exponentially and unbounded growing solution for $t \to \infty$, where 
$j \in \{1,2,3,4\}$. In other words, when $b > d$ the semigroup $S$ does not possess a global attractor in the classical sense and the aformentioned `grow-up' phenomenon occurs; the intuition leading to this outcome is also indicated in \S\ref{s:wellposed}. The previously mentioned stability result of unbounded growth 
now offers us 
deeper insights: if one begins sufficiently close to any of the initial values from the growing solutions \eqref{eq:explsol}, there will also be a grow-up in infinite time. 

We can show this for moderate backscatter $b \in (d,2d)$, which is only weakly destabilizing. 
In this case, the operator $\boldsymbol{B} = -(d\Delta^2 + b \Delta)$ has only one unstable eigenvalue, as noted at the beginning of \S\ref{s:PDE} since the zero mean condition makes no difference. See  Figure~\ref{f:specback}. 
Thus,  as discussed in \S\ref{s:PDE}, the linear operator $L = -(d\Delta^2 + b \Delta) - f P\Id^{\perp}$ determining the spectrum and linear stability of system \eqref{eq:2DNS-2} also possesses only one unstable eigenvalue since $L$ and $\boldsymbol{B}$ have the same eigenvalues. It then follows for $b \in (d,2d)$, that the image of unstable modes under the nonlinearity are stable modes only, i.e., there are no resonant interactions among unstable modes. This is an important element in the proof, and it would be interesting to relax this condition. 

In order to prove Theorem~\ref{t:stab2DNSE}, which is more precisely formulated in Theorem~\ref{thm:stabgrow} below, we consider separately the small- and large-scale parts of a solution whose initial value is sufficiently close to the initial values of \eqref{eq:explsol}.  In case $b < 2d = d\lambda_5$ 
we follow \citep{CFT88} and first 
derive an energy equation without nonlinear terms for the small-scale part. 
This allows to infer exponential decay of the small-scale part and thus 
boundedness in particular. 
Heuristically, $b < 2d$ 
and the absence of nonlinear terms in the energy equation imply that the global behavior of the small scale part of solutions is similar to the case when 
$b<d$ as discussed in \S\ref{s:wellposed}. 
In the second part of the proof, we establish a lower bound for the $L_2$-norm of the large-scale component by leveraging the boundedness of the small-scale part.  
It then turns out that the linear instability due to $b>d$ ensures that the lower bound grows exponentially over time, consequently causing the $L_2$-norm of the large-scale part to exhibit the same behavior. This concludes the proof of Theorem~\ref{thm:stabgrow}. 

{\blue In numerical simulations in turns out that even random initial data appear to behave as predicted by Theorem~\ref{t:stab2DNSE}. In Figure~\ref{f:NSsim} we plot a sample run that illustrates this behaviour. We found similar behaviour for some test cases with $b>2d$, and apparent convergence to steady large scale modes for $b=d$. }

\begin{figure}[t!]
\centering
\includegraphics[trim = 3cm 2cm 3cm 0cm, clip, width=\linewidth]{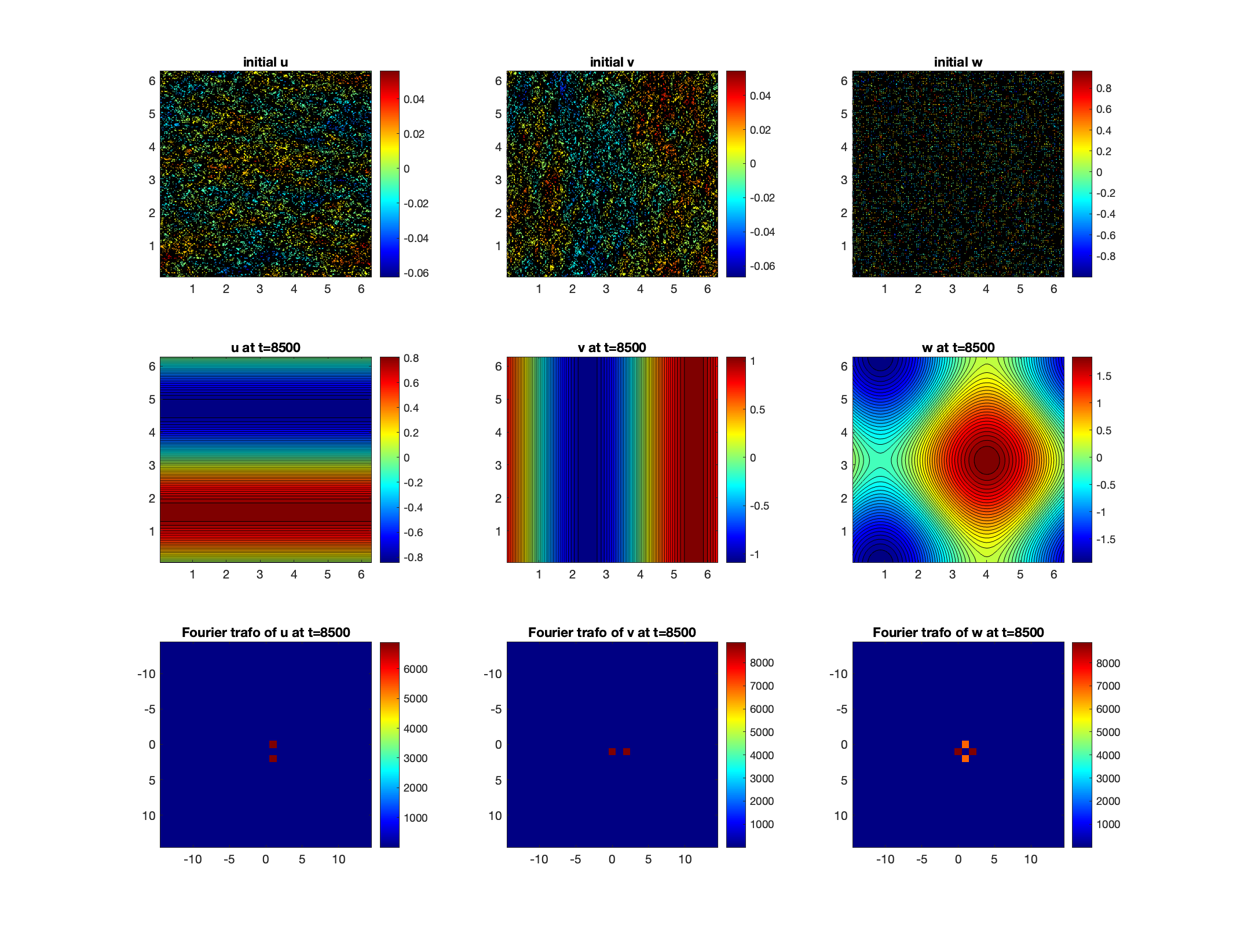}
\caption{\blue Simulation of 2D Euler equations with kinetic energy backscatter $b=0.0015$, $d=0.001$ on $[0,2\pi]^2$ with periodic boundaries by truncated Fourier series with $128\times128$ modes based on \citep{mitcode} with Crank-Nicholson time-step $0.1$. Top rows: initial data and final data at time $t=8500$ showing $u,v$ and vorticity $w$. Bottom row shows absolute values of Fourier coefficients for relatively large scale modes. The final state is essentially a linear combination of $\e_j$, $j=1,\ldots,4$.
% with large prefactor.
}
\label{f:NSsim}
\end{figure}
 
\medskip
In preparation of the theorem formulation and proof, we note that
\begin{align}
\langle B(\u,\v), \v\rangle_2 &= 0 \qquad (\u,\v \in H_{\div}^{0,1}(\T^2,\R^2)), \label{eq:divZero} \\
\langle B(\u,\u), A\u\rangle_2 &= 0 \qquad (\u \in \operatorname{dom}(A)), \label{eq:NonlinLaplaceZero} \\
\langle B(\e_j,\e_k), \e_{\ell}\rangle_2 &= 0 \qquad (j,k,\ell \in \{1,2,3,4\}); \label{eq:EigVectZero}
\end{align}
where 
$\langle \boldsymbol{f},\boldsymbol{g} \rangle_2$ denotes the inner product on $L_2(\T^2,\R^2)$ for $\boldsymbol{f},\boldsymbol{g} \in L_2(\T^2;\R^2)$. The equality \eqref{eq:divZero} is a direct consequence of the fact that $\u$ is divergence free and \eqref{eq:EigVectZero} follows by a straightforward computation. A proof for \eqref{eq:NonlinLaplaceZero} can be found in \citep[Lemma 3.1]{T83}; 
we do not know of a similar result in 3D, which is one obstacle to extend Theorem~\ref{thm:stabgrow} as presented below to  3D models such as the Boussinesq or primitive equations. 

We define the aforementioned `large-scale' component of a function $\u \in L_{2,\div}^0(\T^2;\R^2)$ by
\[
\overline{\u} := \sum_{j=1}^4 \langle \u,\e_j\rangle_2 \e_j
\]
and the `small-scale' component by
\[
\u':= \sum_{j=5}^{\infty} \langle \u,\e_j\rangle_2 \e_j,
\]
where we adhere to the notation commonly used in, e.g., Reynolds averaging;  
notably, $\u = \overline{\u} + \u'$.

With these preparations, we can formulate the announced result, which contains Theorem~\ref{t:stab2DNSE}, as follows.

\begin{thm}\label{thm:stabgrow}
Assume that $b \in (d,2d) =(\lambda_1d,\lambda_5d)$, and let $\u_{\ast} \in \operatorname{lin} \{\e_1,\dots,\e_4\} \setminus \{0\}$. For any $\varepsilon \in (0,1/2)$ there exists $R_{\varepsilon} > 0$ such that for all 
\[
\u_0 \in B_{H^{0,2}}(\u_{\ast},R_{\varepsilon}) := \{ \v \in H_{\textnormal{div}}^{0,2}(\T^2;\R^2); \ \|\u_{\ast} - \v\|_{H^2(\T^2;\R^2)} < R_{\varepsilon} \}
\]
one has that
\begin{equation}\label{eq:expgrowth}
\|S(t)\u_0\|_2 \geq \|\overline{S(t)\u_0}\|_2 \geq \frac{1}{2\sqrt{2}}  \rme^{(1-2\varepsilon)(b-d) t}\|\u_{\ast}\|_2 \qquad (t \geq 0).
\end{equation}
Moreover, there is $c>0$ such that for all $\u_0 \in H_{\textnormal{div}}^{0,2}(\T^2;\R^2)$ one has that
\begin{equation}\label{eq:expdecay}
\|(S(t)\u_0)'\|_{H^1(\T^2;\R^2)} \leq c\,\rme^{-2(2d-b)t}\|\nabla \u_0'\|_2 \qquad (t \geq 0).
\end{equation}
%for some $c > 0$ independent of $u_0$\JYx{$\u_0$?}. 
\end{thm}
\begin{rem}
Concerning the rates in \eqref{eq:expgrowth} and \eqref{eq:expdecay}, recall that $b-d$ is the maximal linear growth rate of the linearization of \eqref{eq:2DNS-3} and $-2(2d-b)$ its slowest decay rate (for mean zero initial data).  
The arbitrarily small correction factor $2\eps$ in \eqref{eq:expgrowth} can be understood similar to the rate correction in un/stable manifolds in the presence of multiple eigenvalues. Establishing an unstable manifold would imply a version of Theorem~\ref{thm:stabgrow} that is local in the sense that it holds as long as $S(t)\u_0$ lies in the stable fibers of the unstable manifold; this would also hold for more general parameter values. However, Theorem~\ref{thm:stabgrow} holds \emph{globally in time}. The growth rate correction means that one cannot generally expect Lyapunov stability or asymptotic stability of the explicit unboundedly growing solutions $S(t)\e_j$, i.e., we cannot expect bounded or asymptotically vanishing 
 $\|S(t)\u_0 - S(t) \u_{\ast} \|_2$ for $\u_{\ast}=\e_j$, $j \in \{2,3,4\}$. 
\end{rem}

\begin{proof}[Proof of Theorem~\ref{thm:stabgrow}]
The proof is divided into three steps. In step (i) we will show that for $\u_0 \in H_{\textnormal{div}}^{0,2}(\T^2;\R^2)$ the `large-scales' 
of a global strong solution to \eqref{eq:2DNS-2} tend to $0$ as $t \rightarrow \infty$ at exponential rate, i.e., we show \eqref{eq:expdecay}. The proof of this follows the same argumentation as in Remark 3.1(ii) of \citep{CFT88}.

(i) Let $\u_0 \in H_{\textnormal{div}}^{0,2}(\T^2;\R^2)$, $\u:= S(\cdot)\u_0$, $\w_0:= \u_0 -\e_1$ and $\w:= \u - \e_1$. 
Then, using the fact that $B(\e_1,\e_1) = 0$, substitution into \eqref{eq:2DNS-2a} gives 
\begin{equation}\label{w-eq-1}
\partial_t \w + d A^2\w - b A\w + B(\w,\e_1) + B(\u,\w) + f\w^{\perp} = 0,
\end{equation}
We take the inner product in $L_2(\T^2;\R^2)$ of \eqref{w-eq-1} with $\w$ and $A\w$, respectively. Combining \eqref{eq:divZero}, \eqref{eq:NonlinLaplaceZero}, with $\langle \w^{\perp},\w\rangle_2 = 0$ and 
\[
\langle \w^{\perp},A\w\rangle_2 = - \sum_{j=1}^2\langle \p_j \w^{\perp},\p_j \w\rangle_2 = 0,
\]
we obtain
\begin{align}
&\frac{1}{2} \frac{\dif}{\dif t} \|\w\|_2^2 + d\|A\w\|_2^2 - b\|A^{1/2}\w\|_2^2 + \langle B(\w,\e_1),\w \rangle_2 = 0, \label{w-eq-2} \\ 
&\frac{1}{2} \frac{\dif}{\dif t} \|A^{1/2}\w\|_2^2 + d\|A^{3/2}\w\|_2^2 - b\|A\w\|_2^2 + \langle B(\w,A\e_1),\w \rangle_2 = 0. \label{w-eq-3} 
\end{align}
Multiplying \eqref{w-eq-2} by $\lambda_1$ and subtracting the resulting equation from \eqref{w-eq-3} then gives
\begin{align}
\frac{1}{2} \frac{\dif}{\dif t} (\|A^{1/2}\w\|_2^2 - \lambda_1\|\w\|_2^2) + 
d(\|A^{3/2}\w\|_2^2 - \lambda_1\|A^{1/2}\w\|_2^2) \label{w-eq-4}  
-b(\|A\w\|_2^2 - \lambda_1\|A^{1/2}\w\|_2^2)=0. 
\end{align}
Using the orthogonal basis we have 
\begin{align*}
d(\|A^{3/2}\w\| - \lambda_1\|A^{1/2}\w\|_2^2) - b(\|A\w\|_2^2 - \lambda_1\|A^{1/2}\w\|_2^2) 
= \sum_{j = 1}^{\infty} \lambda_j(d\lambda_j(\lambda_j - \lambda_1) - b (\lambda_j - \lambda_1)) |\langle \w,e_j \rangle_2|^2,
\end{align*}
and $\lambda_1 = \lambda_2 = \lambda_3 = \lambda_4$ means that the first four summands vanish. 
By the assumption that $b \in (\lambda_1d,\lambda_5d)$, the remaining coefficients satisfy 
\[
\lambda_j(d \lambda_j(\lambda_j - \lambda_1) - b (\lambda_j - \lambda_1)) 
\geq \lambda_5(d \lambda_5 - b) (\lambda_j - \lambda_1) \qquad (j\geq 5).
\] 
Therefore, using again the bases, we obtain the estimate
\begin{align}
d(\|A^{3/2}\w\| - \lambda_1\|A^{1/2}\w\|_2^2) - b(\|A\w\|_2^2 - \lambda_1\|A^{1/2}\w\|_2^2) 
&\geq  \lambda_5(d \lambda_5 - b)(\|A^{1/2}\w\|_2^2 - \lambda_1\|\w\|_2^2). \label{w-ineq-1}
\end{align}
Moreover, since $\langle \w,\e_j \rangle_2 = \langle \u,\e_j \rangle_2$ for all $j \geq 5$ and $\lambda_1 = 1$, $\lambda_5 = 2$, we also have that
\begin{align}
\|A^{1/2}\w\|_2^2 - \lambda_1\|\w\|_2^2 &= \sum_{j=5}^{\infty} (\lambda_j - \lambda_1)|\langle \w,\e_j\rangle_2|^2 \label{w-ineq-2} \\ 
&\geq \biggl(1 - \frac{\lambda_1}{\lambda_5}\biggr) \sum_{j=5}^{\infty} \lambda_j |\langle \u,\e_j\rangle_2|^2 
= \frac{1}{2}\|A^{1/2}\u'\|_2^2. \nonumber
\end{align}
Therefore, we deduce from \eqref{w-eq-4} and \eqref{w-ineq-1} that
\begin{equation}\label{w-ineq-3}
\frac{1}{2} \frac{\dif}{\dif t}(\|A^{1/2}\w\|_2^2 - \lambda_1 \|\w\|_2^2) + \lambda_5(\lambda_5d - b)(\|A^{1/2}\w\| - \lambda_1\|\w\|_2^2) \leq 0.
\end{equation}
Hence, Gronwall's inequality together with~\eqref{w-ineq-2} implies that
\begin{align}
\|A^{1/2}\u'\|_2^2 &\leq 2  (\|A^{1/2}\w\|_2^2 - \lambda_1\|\w\|_2^2)  
  \leq 2 \rme^{-2\lambda_5(d \lambda_5 - b)t}(\|A^{1/2}\w_0\|_2^2 - \lambda_1\|\w_0\|_2^2) \nonumber \\
 &\leq 2\rme^{-4(2d - b)t}\|A^{1/2}\u'_0\|_2^2 \leq 2\|A^{1/2}\u'_0\|_2^2 \qquad (t \geq 0). \label{u-bar-ineq-1}
\end{align}
This gives \eqref{eq:expdecay}.

\medskip
(ii) We will now prove a lower bound for $\|\overline{\u}\|_2^2$, which we will use in step (iii) to show \eqref{eq:expgrowth}. To start with, due to \eqref{eq:EigVectZero} we have
\[
\langle B(\u,\u),\e_j \rangle_2 = \langle B(\overline{\u},\u'),\e_j \rangle_2 + \langle B(\u',\overline{\u}),\e_j \rangle_2 + \langle B(\u',\u'),\e_j \rangle_2 \qquad (j=1,\dots, 4).
\] 
Thus, by projecting \eqref{eq:2DNS-2a} onto the large scale space (i.e.\ by computing the sum of $\langle\cdot, \e_j\rangle_2 \e_j$, $j=1,\ldots,4$ for each term) we find that
\begin{equation}\label{u-tilde-eq-1}
\partial_t \overline{\u} + d A^2\overline{\u} - b A\overline{\u} + \sum_{j=1}^4 (\langle B(\overline{\u},\u'),\e_j \rangle_2 + \langle B(\u',\overline{\u}),\e_j \rangle_2 + \langle B(\u',\u'),\e_j \rangle_2)\e_j  + f\overline{(\u^{\perp})} = 0.    
\end{equation}
We next take the inner product in $L_2(\T^2;\R^2)$ of \eqref{u-tilde-eq-1} with $\overline{\u}$. By \eqref{eq:divZero}, we obtain
\begin{equation*}
\frac{1}{2} \frac{\dif}{\dif t} \|\overline{\u}\|_2^2 + d \|A\overline{\u}\|_2^2 - b \|A^{1/2}\overline{\u}\|_2^2 + \langle B(\overline{\u},\u'),\overline{\u}\rangle_2 + \langle B(\u',\u'),\overline{\u}\rangle_2 + f\langle \overline{(\u^{\perp})},\overline{\u} \rangle_2 = 0,
\end{equation*}
which, due to $\lambda_1 = \lambda_2 = \lambda_3 = \lambda_4$, is equivalent to
\begin{equation}\label{u-tilde-eq-2}
\frac{1}{2}\frac{\dif}{\dif t}\|\overline{\u}\|_2^2 + \lambda_1(d \lambda_1 - b) \|\overline{\u}\|_2^2 + \langle B(\overline{\u},\u'),\overline{\u}\rangle_2 + \langle B(\u',\u'),\overline{\u}\rangle_2 + f\langle \overline{(\u^{\perp})},\overline{\u} \rangle_2 = 0  
\end{equation}
Since $\|A^{1/2}\v\|_2^2 = \int_{\T^2}  |\nabla \v|^2 \,\dif\x$ for all $\v \in H_{\div}^{0,1}(\T^2;\R^2)$ and the fact that $|\e_j| \leq 1 \ (j = 1,\dots,4)$  (see \eqref{eq:eigfunct}) implies that $\|\overline{\v}\|_{\infty} \leq 2\|\overline{\v}\|_2$ for all $\v \in L_{\div,2}^0(\T^2;\R^2)$, we can estimate
\begin{align}
&|\langle B(\overline{\u}, \u'),\overline{\u} \rangle_2| \leq \|\overline{\u}\|_{\infty} \|\nabla \u'\|_2 \|\overline{\u}\|_2 \leq \label{nonlin-eq-1} 2\|A^{1/2}\u'\|_2 \|\overline{\u}\|_2^2, \\
&|\langle B(\u', \u'),\overline{\u} \rangle_2| \leq \|\u'\|_2\|\nabla \u'\|_2 \|\overline{\u}\|_{\infty} \leq 2\lambda_1^{-1/2}\|A^{1/2}\u'\|_2^2 \|\overline{\u}\|_2. \label{nonlin-eq-2}
\end{align}
Moreover, since $\overline{\v_1}$ and $\v_2'$ are orthogonal in $L_2(\T^2;\R^2)$ and $\langle \v^{\perp},\v \rangle_2 = 0$ for all $\v,\v_1,\v_2 \in L_{2,\div}^0(\T^2;\R^2)$, we have
\begin{equation}\label{eq:rotest}
|\langle \overline{(\u^{\perp})},\overline{
\u} \rangle_2| = |\langle \u^{\perp},\overline{
\u} \rangle_2| = |\langle (\u')^{\perp},\overline{\u}\rangle_2| \leq \|\u'\|_2\|\overline{\u}\|_2 \leq \lambda_1^{-1/2}\|A^{1/2}\u'\|_2 \|\overline{\u}\|_2    
\end{equation}
Hence, we deduce from \eqref{u-tilde-eq-2}, \eqref{nonlin-eq-1} \eqref{nonlin-eq-2}, \eqref{eq:rotest}, \eqref{u-bar-ineq-1}, the fact that $\lambda_1 = 1$ and the Peter-Paul inequality (i.e. $a_1a_2 \leq \frac{1}{2}(\gamma a_1^2 + \gamma^{-1} a_2^2)$ for all $a_1,a_2 \in \R$, $\gamma > 0$) that
\begin{align*}
&\frac{1}{2}\frac{\dif}{\dif t} \|\overline{\u}\|_2^2 + \biggl((1-\varepsilon)(d - b) + 2\sqrt{2}\|A^{1/2}\u'_0\|_2\biggr)\|\overline{\u}\|_2^2  \\ 
&\geq - \frac{2}{\varepsilon(b - d)}\biggl(4 \|A^{1/2} \u'_0\|_2^2 + \frac{f^2}{2} \biggr)\|A^{1/2} \u'_0\|_2^2
\end{align*}
for all $\varepsilon > 0$. Defining
\[
\omega_{\varepsilon}(r):= 2(1-\varepsilon)(d - b) + 4\sqrt{2}r, \qquad 
c_{\varepsilon}(r):= \frac{4}{\varepsilon(b - d)}\biggl(4r^2 + \frac{f^2}{2} \biggr)r^2, 
\qquad (\eps>0, \; r \geq 0),
\]
we infer, if  $\omega_{\varepsilon}(\|A^{1/2}\overline{\u_0}\|_2) \neq 0$, the estimate (for all $t\geq 0$)
\begin{equation}\label{u-tilde-ineq-1}
\|\overline{\u}\|_2^2 \geq \rme^{-\omega_{\varepsilon}(\|A^{1/2}\u'_0\|_2) t}\|\overline{\u_0}\|_2^2 + \frac{c_{\varepsilon}(\|A^{1/2}\u_0'\|_2)}{\omega_{\varepsilon}(\|A^{1/2}\u'_0\|_2)}(\rme^{-\omega_{\varepsilon}( \|A^{1/2}\u'_0\|_2)t} - 1).
\end{equation}

\medskip
(iii) Using step (i) and (ii), we next show that \eqref{eq:expgrowth} holds, which completes the proof. We set $\gamma:=\|\u_{\ast}\|_2^2 \ (> 0)$ and fix $\varepsilon \in (0,1/2)$. It is straightforward that there exists a $r_{\varepsilon}' > 0$ with 
\begin{equation}\label{c-omega-gamma-est}
\omega_{\varepsilon}(r) < 0 ,\qquad \frac{c_{\varepsilon}(r)}{\omega_{\varepsilon}(r)} \geq - \frac{\gamma}{16} \qquad (0 \leq r \leq r_{\varepsilon}').
\end{equation}
We now set $r_{\varepsilon}:= \min\{r_{\varepsilon}',\varepsilon(b-d)/2\sqrt{2}\}$ and 
$R_\eps:= \min \{r_{\varepsilon}/2,\gamma^{1/2}/4\}$, so that by assumption
\begin{equation}\label{u_0-assump}
\u_0 \in B_{H^{0,2}}(\u_{\ast},\min \{r_{\varepsilon}/2,\gamma^{1/2}/4\}).
\end{equation}
This implies 
\[
\|\u_{\ast}\|_2^2 \leq 4\|\overline{\u_0}\|_2^2 + \|\u_{\ast} - \overline{\u_0}\|_2^2  \leq 4\|\overline{\u_0}\|_2^2 + 4\|\u_{\ast} - \u_0\|_2^2 \leq 4\|\overline{\u_0}\|_2^2 + \frac{\gamma}{4}
\]
and hence
\begin{equation}\label{w_0-ineq-2}
\frac{1}{4} \biggl(\|\u_{\ast}\|_2^2 - \frac{\gamma}{4}\biggr) \leq \|\overline{\u_0}\|_2^2.    
\end{equation}
Moreover, \eqref{u_0-assump} yields that 
\begin{align}
\|A^{1/2}\u'_0\|_2 = \|A^{1/2}\u_0 - A^{1/2}\overline{\u_0}\|_2  &\leq \|A^{1/2}(\u_0 - \u_{\ast})\|_2 + \|A^{1/2}(\u_{\ast} - \overline{\u_0})\|_2 \nonumber \\ &= 2\|A^{1/2}(\u_0 - \u_{\ast})\|_2 \leq r_{\varepsilon}. \label{u-bar-ineq-2}
\end{align}
Therefore, using \eqref{u-tilde-ineq-1} with \eqref{c-omega-gamma-est}, and then \eqref{w_0-ineq-2} as well as \eqref{u-bar-ineq-2}, we infer for all $t \geq 0$ that

\begin{align*}
\|\overline{\u}\|_2^2 &\geq \rme^{-\omega_{\varepsilon}(\|A^{1/2}\u'_0\|_2) t}\|\overline{\u_0}\|_2^2  + \frac{c_{\varepsilon}(\|A^{1/2}\u_0'\|_2)}{\omega_{\varepsilon}(\|A^{1/2}\u_0'\|_2)}(\rme^{-\omega_{\varepsilon}( \|A^{1/2}\u'_0\|_2)t} - 1), \tag*{} \\
&\geq \rme^{-\omega_{\varepsilon}(\|A^{1/2}\u'_0\|_2) t}\biggl(\|\overline{\u}_0\|_2^2 - \frac{\gamma}{16} \biggr) \\ 
&\geq \frac{1}{4}\rme^{-\omega_{\varepsilon}(\|A^{1/2}\u'_0\|_2) t} \biggl(\|\u_{\ast}\|_2^2 -  \frac{\gamma}{2}\biggr) \geq \frac{1}{8} \|\u_{\ast}\|_2^2 \rme^{-\omega_{\varepsilon}(r_0)t} \geq \frac{1}{8} \|\u_{\ast}\|_2^2 \rme^{2(1-2\varepsilon)(b-d)t}
\end{align*}
where in the last step we used the inequality $\omega_{\varepsilon}(r_{\varepsilon}) \leq \omega_{\varepsilon}(\varepsilon(b-d)/2\sqrt{2}) = 2(1-2\varepsilon)(d-b)$. (Also recall $\gamma=\|\u_{\ast}\|_2^2$.)
This completes the proof.
\end{proof}

\section{Bottom drag and bifurcations}\label{s:bottom}

In this section we review results from \citep{PRY2023} concerning the impact of combined backscatter and bottom drag in \eqref{e:sw}. For nonlinear drag $Q\neq 0$ mass conservation implies that unbounded growth cannot occur. It turns out that the growth of $\vh$ can saturate in solutions that bifurcate from the trivial steady flow $(\vh,\eta)= (0,0)$ at a critical linear drag parameter $\lb$. However, as mentioned, due to the missing spectral gap {\blue at large wavenumbers} from \eqref{e:speczero}, the bifurcation problem is non-standard{, and one cannot expect a finite dimensional reduction of the entire dynamics even on spaces of periodic functions}. {\blue Nevertheless, the specific bifurcations we review here can still be rather directly analyzed using Lyapunov-Schmidt reduction.  
%as detailed in \citep{PRY2023}. 
Briefly, in \S\ref{s:bif} below we consider steady bifurcations that occur in an invariant subspace within which the spectrum is parabolic. For the non-steady bifurcations considered in \S\ref{s:igw}  Fredholm properties can be established on spaces of periodic functions as detailed in \citep{PRY2023}. 
Notably, we do not claim that arbitrary perturbation from the trivial state converges to these bifurcating flows.}

Initial insight into the impact of combined backscatter and bottom drag can be gained from the ansatz \eqref{sol: RSWB1}. 
The analogue of \eqref{cond: RSWB1} in case of isotropic backscatter $d=d_1=d_2$, $b=b_1=b_2$ is
\begin{subequations}\label{e:plane}
\begin{align}
f\psi &= g\partial_\xi\phi,\label{e:plane2a}\\
\partial_t \psi &= -d k^4\partial_\xi^4\psi - b k^2\partial_\xi^2\psi - \frac{C+Qk|\psi|}{H_0+\phi}\psi, \label{e:plane2b}
\end{align}
\end{subequations}
where $k:=|\k|$ is the wave number. Here \eqref{e:plane2a} is the geostrophic balance relation \eqref{cond: RSWB1b} in the isotropic case, but \eqref{e:plane2b} contains the nonlinearity resulting from the drag $\F$ for $Q\neq 0$ or $C\neq 0$.
\begin{rem}[Anisotropy]\label{r:anisoscalar}
In the anisotropic case $b_1\neq b_2$ or $d_1\neq d_2$, 
we can still reduce \eqref{e:sw} to \cref{e:plane}, when replacing $b,d$ by $b_2,d_2$ (or $b_1,d_1$) and choosing the wave vector $\k=(1,0)$ (or $\k=(0,1)$) in \cref{sol: RSWB1}.
\end{rem}
From \cref{e:plane2a} it follows that also $\psi$ is independent of $t$ unless $f=0$. For $f=0$ we have $\phi=\phi_0\in\R$ is constant, and \cref{e:plane2b} has the form of a non-smooth quadratic Swift-Hohenberg-type equation with parameter $\phi_0$. 
This naturally admits a weakly nonlinear analysis near the trivial state $\psi=0$ within the specific class of solutions. For $Q=0$ we can employ \eqref{sol: RSWB2} with $\alpha_2=0$ and without loss mean zero $s=0$. In particular, unbounded growth still occurs for $Q=0$ and moderate linear drag $\lb>0$; although for $f\neq0$ this requires anisotropic backscatter. We plot some resulting solution loci in Figure~\ref{f:growth} and refer the reader to \citep{PRY2023} for details.

\begin{figure}[t!]
\centering
\subfigure[$\lb=0.11$]{\includegraphics[trim = 5.5cm 8.5cm 6cm 8.5cm, clip, width=0.24\linewidth]{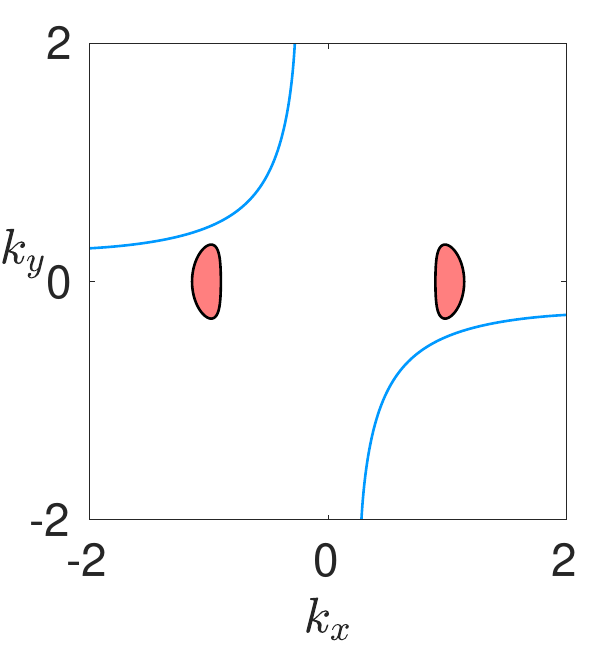}\label{f:anisoC0p11}}
\hfil
\subfigure[$\lb=0.08$]{\includegraphics[trim = 5.5cm 8.5cm 6cm 8.5cm, clip, width=0.24\linewidth]{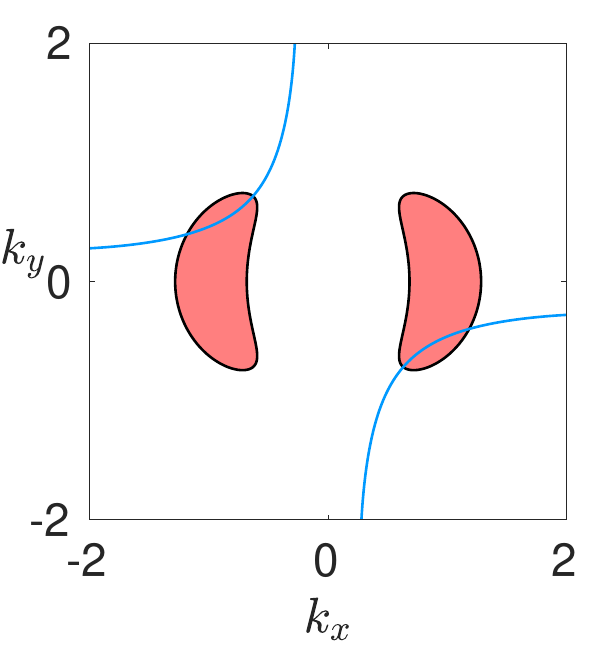}\label{f:anisoC0p08}}
\hfil
\subfigure[$\lb=0.05$]{\includegraphics[trim = 5.5cm 8.5cm 6cm 8.5cm, clip, width=0.24\linewidth]{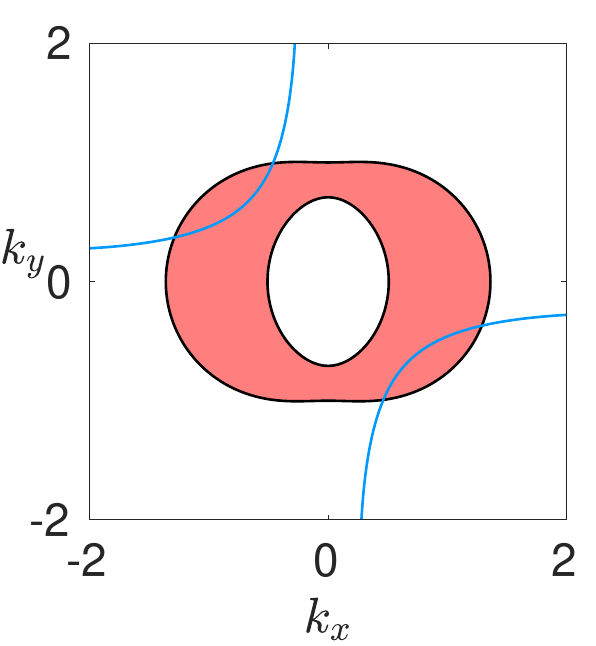}\label{f:anisoC0p05}}
\hfil
\subfigure[$\lb=0$]{\includegraphics[trim = 5.5cm 8.5cm 6cm 8.5cm, clip, width=0.24\linewidth]{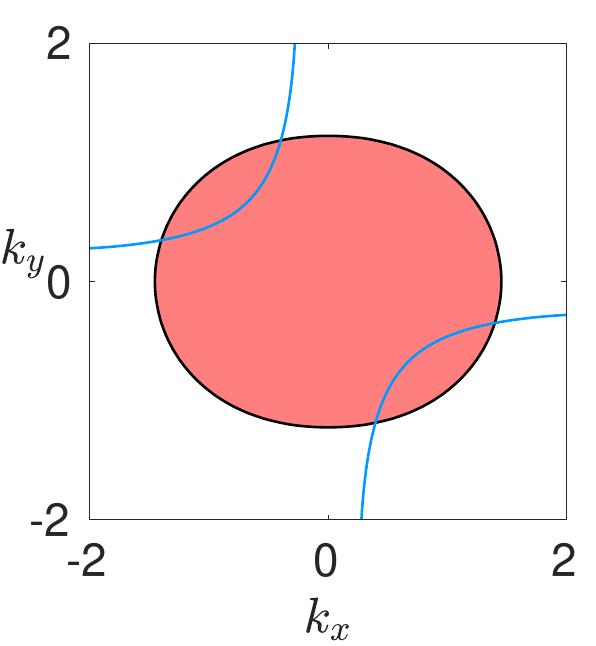}\label{f:anisoC0}}
\caption{(Fig.~4 in \citep{PRY2023}.) Samples of the loci of explicit flows \eqref{sol: RSWB2} (red: $\lambda>0$; white: $\lambda<0$; black: $\lambda=0$ (except $\k=0$)) with $\alpha_2=s=0$, $d_1=1$, $d_2=1.04$, $b_1=1.5$, $b_2=2.2$, $f=0.3$, $g=9.8$, $H_0=0.1$, $Q=0$ so that $\lb_\crit\approx0.116$. Blue: the wave vectors satisfying \eqref{cond: RSWB2b} with $\alpha_2=0$. The intersections of blue and black curves are loci of steady flows, those within red regions are loci of unboundedly growing flows.}
\label{f:growth}
\end{figure}

\subsection{Onset of instabilities}\label{s:Turing}

As discussed in \S\ref{s:spec} the onset of instability occurs at $\lb=\lb_\crit$ with critical wavenumber $\k=\k_\crit$ from \eqref{e:isocritCk}, and the critical modes can be interpreted as geostrophically balanced and  gravity waves. This is also reflected in the eigenvectors of the matrix $\hat\calL$; in this matrix only the first two diagonal entries depend on backscatter and in the present isotropic case these entries are $-F$ as defined before \eqref{e:isocritCk}. 
Hence, the critical case $F=0$ gives the same matrix as the inviscid case and thus the same eigenvectors. 

For single steady modes, that lead to geostrophic equilibria, due to the rotation symmetry any wave vector direction can be selected in order to reduce \eqref{e:sw} to a 1D plane wave problem. This always gives \cref{e:plane} and admits an analysis of the stationary bifurcation problem independent of the spectral gap problem.  
We choose $\x = (x,0)$ on the rescaled domain $[0,2\pi]$ and restrict $\calL$ to this 1D situation. 
The kernel eigenvectors of $\calL_\crit$ at $k = k_\crit$ are 
$\tee_j:=\tE_j \rme^{\rmi j x}$, $\tE_j\in\C^3$ and $\tE_{-j} = \overline{\tE}_j$, for $j=0,\pm1$, and 
\begin{equation}\label{e:rossbyeigenvector}
\tE_j = 
\begin{pmatrix}
0\\ 1\\ -\rmi j f/(g k_\crit)
\end{pmatrix},
\; j=\pm 1, \quad 
\tE_0 = 
\begin{pmatrix}
0\\ 0\\ 1
\end{pmatrix}.
\end{equation}
As a consequence of mass conservation, {\blue $\tee_0=\tE_0$} is in the kernel of $\calL$ for all $\lb$. The first two components of $\tE_{\pm 1}$ are orthogonal to $\k=(k_x,0)$, which is in line with the reduction to \cref{e:plane}, cf.\ Remark \cref{r:anisoscalar}. These eigenvectors depend on the bottom drag parameters only through $k_\crit$ from \cref{e:isocritCk}, which can be written as $\lb_\crit = d H_0k_\crit^4$. As mentioned, $\calL$ for $d=b=\lb=0$ possesses the same kernel eigenvectors, even for general $\k$, and the kernel eigenvectors $\tee_{\pm1}$ correspond to geostrophically balanced modes.

Oscillatory modes become steady in a comoving frame $\x - \c t$ with suitable constant $\c$. 
This change of coordinates creates the additional term $-\c\cdot\nabla (\vh,\eta)$ on the left-hand side of \cref{e:sw} so that the dispersion relation \cref{e:disp} turns into 
\begin{equation}\label{e:dispcomov}
\disp_{\c}(\lambda, \k) = \disp(\lambda - \rmi\c \cdot\k,\k). 
\end{equation}
This means that exactly the purely imaginary spectrum at $C=\lb_\crit$ from \cref{e:isocritCk} is shifted to the origin for $\omega =\pm\omega_\crit$ with $\c=-\omega \k_\crit/|\k_\crit|^2$ and any $\k_\crit = (\kcx,\kcy)$ satisfying $|\k_\crit|=k_\crit$ from \eqref{e:isocritCk}. Hence, a 1D reduction to $2\pi$-periodicity yields the phase variable 
\[
\zeta=(\x-\c t)\cdot\k_\crit = \kcx x+\kcy y + \omega t.
\]
and $\partial_x = \kcx\partial_\zeta$, $\partial_y=\kcy\partial_\zeta$, and $\partial_t$ becomes $\partial_t + \omega\partial_\zeta$ so that $\calL_\crit$ in terms of $\zeta$ becomes
\begin{align}\label{e:HopfOp}
\calL_\crit = \begin{pmatrix}
\omega_\crit\partial_\zeta- d k_\crit^4(\partial_\zeta^2+1)^2 & f & -\kcx g \partial_\zeta\\
-f & \omega_\crit\partial_\zeta- d k_\crit^4(\partial_\zeta^2+1)^2 & -\kcy g \partial_\zeta\\
-\kcx H_0 \partial_\zeta & -\kcy H_0 \partial_\zeta  & \omega_\crit\partial_\zeta
\end{pmatrix}.
\end{align}
Here we used 
$k_\crit^4 d \partial_\zeta^4 + k_\crit^2 b \partial_\zeta^2 + \lb_\crit/H_0 = d k_\crit^4(\partial_\zeta^2+1)^2$  
for $\omega=-\omega_\crit$. 
The kernel of this $\calL_\crit$ is also three-dimensional, spanned by $\ee_0=\tee_0$ from \cref{e:rossbyeigenvector} and 
$\ee_j:=\E_j \rme^{\rmi j \zeta}$, $\E_j\in\C^3$, $j=\pm1$, where 
\[
\E_j = \begin{pmatrix}
\omega_\crit\kcx + j\rmi f\kcy \\
\omega_\crit\kcy - j\rmi f\kcx \\
k_\crit^2 H_0
\end{pmatrix}.
\]
The case $\omega=\omega_\crit$ is analogous and again the eigenvectors $\ee_{\pm1}$ depend on the bottom drag only through $\k_\crit$. 
The structure of $\E_j,\,j=\pm 1$ is inconsistent with the plane wave form \cref{sol: RSWB1} so that a bifurcation analysis requires to consider the full system. 
Consistent with the above interpretation, these kernel eigenvectors correspond to gravity waves.

\subsection{Bifurcation of nonlinear geostrophic equilibria}\label{s:bif}

Bifurcations arise from the kernel of $\calL_\crit$ that was discussed above. We next assume isotropic backscatter, which is without loss concerning bifurcations from the purely steady or purely oscillatory eigenmodes, but not combinations. Regarding steady state bifurcations, using \eqref{e:plane} in \eqref{sol: RSWB1} we thus look for geostrophic equilibrium (GE) solutions to \cref{e:sw} of the form 
\begin{equation}\label{e:special}
\vh = \partial_\xi \phi(\k\cdot\x)\k^\perp,\quad \eta =\frac f g\phi(\k\cdot\x),\quad |\k| \approx k_\crit,
\end{equation}
for wave shape $\phi$ and phase variable $\xi=\k\cdot\x$.  
This ansatz reduces \eqref{e:plane} to the steady state equation
\begin{align}\label{e:GE}
0=d k^4\partial_{\xi}^5\phi+b k^2\partial_{\xi}^3\phi+\frac{C+Qk|\partial_\xi\phi|}{H_0+ f\phi/g}\partial_{\xi}\phi, \quad  k=|\k| .
\end{align}
With the bifurcation parameter $\alpha$ such that $\lb = \lb_\crit-\alpha H_0$ we include changes in wave number such that $|\k| = k_\crit+\kappa$ with $\kappa\approx 0$ in the parameter vector $\mu = (\alpha,\kappa)\approx(0,0)$ and expand $\lambda = \lambda(\mu)$ solving \cref{e:disp} with $\lambda(0) = 0$ 
as
\begin{equation}\label{e:zerospec}
\lambda(\mu) =  M \big(\alpha - 2b\kappa^2 + \calO(|\kappa|^3)\big), 
\end{equation}
where 
$M:= gk_\crit^2H_0/\omega_\crit^2>0$. Hence, the trivial state $(\vh,\eta)=(0,0)$ (i.e.\ $\phi(\xi)\equiv0$) is unstable for $\alpha > 2b\kappa^2$ and the stability boundary is given by $\alpha = 2b\kappa^2$ at leading order. Since the nonlinear drag force $\F$, i.e., the nonlinearity in \eqref{e:GE} is not smooth for $Q\neq0$, we need to distinghuish $Q=0$ and $Q\neq0$. However, it turns out that we can use the same method to derive bifurcation equations. 

\begin{thm}[Bifurcation of GE \citep{PRY2023}]\label{t:bifQ0}
Let $\alpha,\kappa\in\R$ be sufficiently close to zero, and let $\k\in\R^2$ with $|\k| =k_\crit+\kappa$. 
Consider steady plane wave-type solutions to \cref{e:sw} of the form \cref{e:special} 
with $2\pi$-periodic mean zero $\phi$ and $\xi=\k\cdot\x$, i.e., solutions to \eqref{e:GE}.

These are (up to spatial translations) in one-to-one correspondence with solutions $A_1\in\R$ near zero of
the following algebraic equations:
\begin{align}
&Q=0: \quad 0 = A_1\left(
\alpha  -  2b\kappa^2  - \frac{17b^2 f^2}{72 d g^2 H_0^2}A_1^2 + \calR_\s \right),\label{e:bifQ0}\\
&Q\neq 0: \quad 0 = A_1\left( 
\alpha - 2b\kappa^2 - \frac{16 Q \sqrt b}{3\pi H_0 \sqrt{2d}}|A_1| +  \calR_\ns \right),\label{e:bifQn0}
\end{align}
with remainder terms 
\begin{align*}
\calR_\s &= f \calO(|A_1|^2 (|A_1|^2+|\alpha| +|\kappa|)) +\calO(|\kappa|^3),\\
\calR_\ns &= Q f \calO(|A_1|^2)+(Q+f)\calO(|A_1|(|A_1|+|\alpha|+|\kappa|))+\calO(|\kappa|^3).
\end{align*} 

\noindent In addition, the profiles $\phi(\xi) = \phi_\s(\xi;\mu)$ for $Q=0$ and $\phi(\xi) = \phi_\ns(\xi;\mu)$ for $Q\neq 0$ satisfy
\begin{align*}
\phi_\s(\xi;\mu) &= 2A_1\cos(\xi) + \frac{f}{9 g H_0}A_1^2 \cos(2\xi) + f\calO\left(|A_1|^2 (|A_1| +|\alpha|+|\kappa|)\right),\\
\phi_\ns(\xi;\mu) &= 
2A_1\cos(\xi) + \calO(|A_1|(|A_1| + |\alpha| +|\kappa|)).
\end{align*}
\end{thm}

\paragraph{Smooth case $Q=0$.} 
The bifurcation is a supercritical pitchfork for $f\neq0$, since the coefficient of $A_1^3$ in \cref{e:bifQ0} is negative and the zero state is unstable for $\alpha > 2b\kappa^2$, cf.\ Figure \ref{f:ampalp}. The leading order amplitude of bifurcating states is given by 
\[
|A_1| = \frac{6gH_0}{b|f|}\sqrt{\frac{2 d}{17}(\alpha - 2b \kappa^2 )}.
\]
On the one hand, for fixed $0<\alpha\ll1$ the amplitudes $A_1=A_1(\kappa)$, for $\kappa\in(-\kappa_0,\kappa_0)$ with $\kappa_0:=\sqrt{\alpha/(2b)}$, form a semi-ellipse with maximum at $\kappa=0$. On the other hand, amplitudes are proportional to $\sqrt{d\alpha}$ and inversely proportional to $bf$. In particular, existence requires nonzero backscatter. 
The amplitude diverges pointwise in the following cases: (1) as $b\to 0$ for each $\kappa\in(-\kappa_0,\kappa_0)$ (where $\kappa_0$ remains constant if the drag $\alpha$ scales as the backscatter term $b$); (2) for \emph{equally} small backscatter parameters $(d,b)=(\eps\hat d,\eps\hat b)$, as $\eps\to 0$, with $|A_1| = \calO(\eps^{-1/2})$ (the wave number is fixed at $k=k_\crit = (\hat b/(2\hat d))^{1/2}$); (3) as $f\to 0$, as illustrated in Figure \ref{f:amp} ($\eta=f\phi/g$ is bounded). 

At $f=0$, the leading order bifurcation equation is given by $0 = A_1(\alpha - 2b\kappa^2)$, 
so that the value of $A_1$ can be arbitrary at $\alpha = 2b\kappa^2$ (or at $\kappa=\pm\kappa_0$), forming `vertical' bifurcating branches, cf.\ Figure \ref{f:amp}. For these solutions the wave shapes are sinusoidal $\phi_\s(\xi) = 2A_1\cos(\xi)\in N_0$, which is consistent with linear spaces of plane waves from \Cref{s:swe} with $\eta = f \phi_\s=0$ since $f=0$, where the amplitude of the velocity $\vh = \partial_\xi\phi_\s(\xi)\k^\perp = -2A_1\sin(\xi)\k^\perp$ is free. 
\begin{figure}[t!]
\centering
\subfigure[$\kappa=0$]{\includegraphics[trim = 5cm 8.5cm 5.5cm 8.5cm, clip, height=4cm]{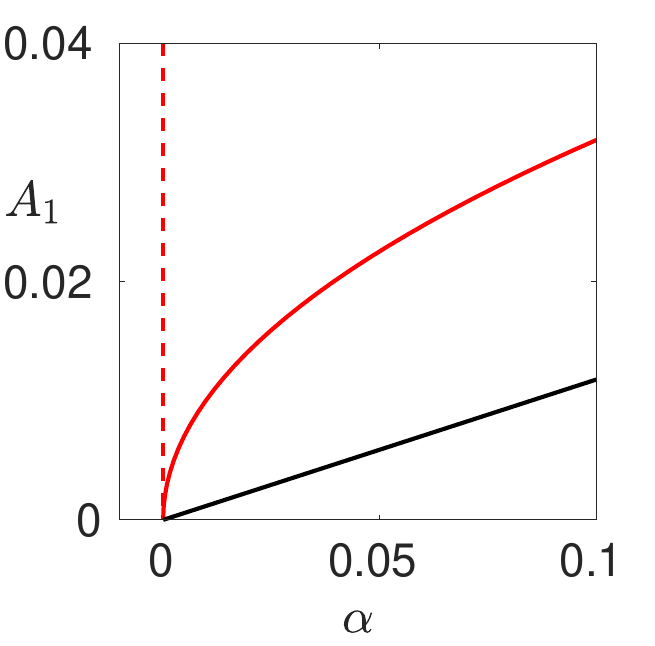}\label{f:ampalp}}
\hfil
\subfigure[$\alpha = 0.1$]{\includegraphics[trim = 5cm 8.5cm 5.5cm 8.5cm, clip, height=4cm]{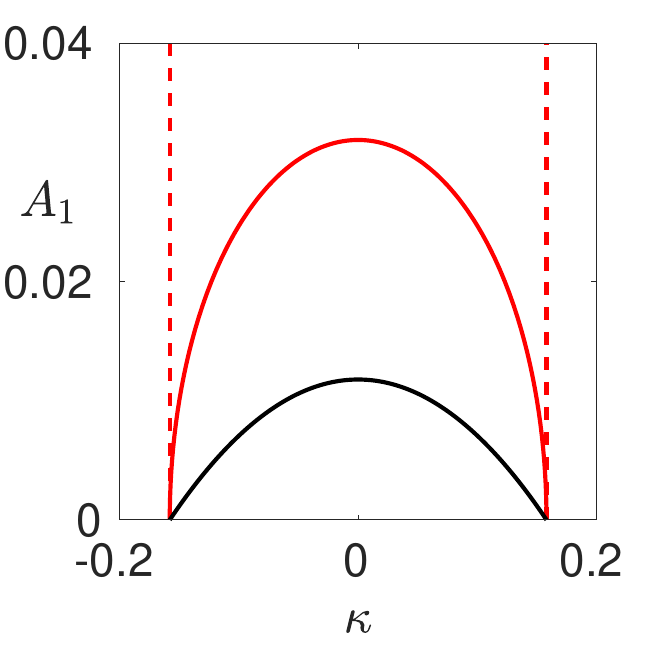}\label{f:ampkap}}
\caption{(Fig.~3 in \citep{PRY2023}.) Leading order amplitudes in the isotropic case $d_1=d_2=1$, $b_1=b_2=2$ for smooth ($Q=0$, red) and non-smooth ($Q=0.5$, black) cases. Other parameters are $g=9.8$, $H_0=0.1$ so that $\lb_\crit = 0.1$. In the smooth case, the amplitudes are plotted for $f=10$ (solid) and $f=0$ (dashed); in the non-smooth case, the amplitude is to leading order independent of $f$.}
\label{f:amp}
\end{figure}

\paragraph{Non-smooth case $Q\neq0$} 
As in the smooth case the bifurcation is always supercritical since the coefficient of $A_1|A_1|$ is negative and the zero state is unstable for $\alpha>2b\kappa^2$, although for $Q\neq0$ it is a degenerate pitchfork, where the bifurcating branch behaves linearly in $\alpha$ near zero, cf.\ Figure \ref{f:ampalp}. 
Different from the smooth case, the coefficient of $A_1|A_1|$ in \cref{e:bifQn0} is independent of Coriolis parameter $f$, which means that the `vertical branch' does not occur. The leading order amplitude of bifurcating solutions for $Q\neq 0$ is given by
\[
|A_1| = \frac{3\pi H_0 \sqrt{2d}}{16Q \sqrt b} \left( \alpha - 2b\kappa^2 \right), 
\]
which is parabolic in $\kappa$, cf.\ Figure \ref{f:ampkap}, and is inversely proportional to $Q$. In contrast to the smooth case, for equally small backscatter parameter $(d,b)=(\eps\hat d, \eps\hat b)$ as $\eps\to 0$, the prefactor is independent of $\eps$. 

\begin{rem}
As $Q\rightarrow0$ in \eqref{e:bifQn0} the nonlinear term in $|A_1|$ vanishes and the remainder term limits to $\calR_\ns = f\calO(|A_1|(|A_1|+|\alpha|+|\kappa|))+\calO(|\kappa|^3)$. 
Compared with the bifurcation equation \eqref{e:bifQ0} at $Q=0$ there may be additional terms of order $f\calO(|A_1|(|\alpha|+|\kappa|))$. 
We expect that it can be shown with a refined analysis that such terms do not occur, which would prove that the bifurcation equation \eqref{e:bifQn0} is continuous with respect to $Q$ at $Q=0$. 
\end{rem}

\subsection{Bifurcation of nonlinear gravity waves}\label{s:igw}
We recall from \Cref{s:spec} that in the isotropic case pattern forming stationary and oscillatory modes destabilize simultaneously. For weak anisotropy oscillatory modes destabilize slightly after the steady ones as $C$ decreases, and the bifurcations of GWs studied next are perturbed as shown numerically in \Cref{s:num}. 
For brevity, we consider only the non-smooth case $Q\neq 0$. 

To include deviations $\kappa$ from the critical wave number $k_\crit$ and $s$ from the critical wave speed $-\omega_\crit$ we extend \cref{e:HopfOp} by taking $\k_\crit=(\kcx,\kcy)$ satisfying \cref{e:isocritCk}, 
 change coordinates to $\zz,\tilde \zz\in \R$ via  
\begin{equation}\label{e:xphase}
\x=\big((1+\kappa)^{-1}\zz+st\big)/|\k_\crit|^2\;\k_\crit + \tilde \zz \k_\crit^\perp,
\end{equation}
and seek solutions that are independent of $\tilde \zz$ for $\mu=(\alpha,s,\kappa)\approx 0$. 
\begin{thm}[Bifurcation of GWs for $Q\neq 0$ \citep{PRY2023}]\label{t:bifIGWQn0}
Let $Q\neq0$, $\alpha,\kappa\in \R$ be sufficiently close to zero, and let $\k_\crit=(\kcx,\kcy)$ be arbitrary with $|\k_\crit|=k_\crit$. 
Consider $2\pi$-periodic steady travelling wave-type solutions $(\vh,\eta)$ to \cref{e:sw} with mean zero $\eta$ and $\zeta=(1+\kappa)^{-1}\zz-(\omega_\crit-s)t$. These are (up to spatial translations) in one-to-one correspondence with solutions $s, A_1$ near zero of 
\begin{subequations}\label{e:gwBifEqu}
\begin{align}
0 &= A_1\left(-s + \kappa \frac{f^2}{\omega_\crit} 
+ \calO(
|A_1|^2+|\mu|^2) 
\right), \label{e:gwBifEqua}\\
0&= A_1\left(\alpha 
- \frac{2 Q  k_\crit}{ H_0} \left( I_1 + \frac{k_\crit^2 g H_0}{2f^2+ k_\crit^2 g H_0} I_2\right)|A_1|+ 
\calO(
|A_1|^2+|\mu|^2) 
\right),\label{e:gwBifEqub}
\end{align}
\end{subequations}
with positive quantities
\begin{align}\label{e:gwCoefficient}
I_1 = \frac{1}{2\pi}\int_0^{2\pi}\sqrt{f^2 + k_\crit^2 g H_0 \cos(\zz)^2}\,\dif \zz,\quad
I_2 =  \frac{1}{2\pi}\int_0^{2\pi}\sqrt{f^2 + k_\crit^2 g H_0 \cos(\zz)^2}\cos(2\zz)\,\dif \zz.
\end{align}
Taking $\x$ as in \cref{e:xphase}, these waves have the form
\[
\begin{pmatrix}\vh\\\eta\end{pmatrix}(t,{\x})=
2A_1\begin{pmatrix}
\omega_\crit \kcx\cos\zeta-f\kcy\sin\zeta\\
\omega_\crit \kcy\cos\zeta + f\kcx\sin\zeta\\
k_\crit^2H_0\cos\zeta
\end{pmatrix}
+\calO(|A_1|(
|A_1|+|\mu|)).
\]
\end{thm}
Here \cref{e:gwBifEqua} determines the deviation by $s$ from the travelling wave velocity and highlights wave dispersion, where these nonlinear GWs of different wave number travel at different velocity. The amplitude $|A_1|$ is determined by \cref{e:gwBifEqub}. Similar to GE, the bifurcation is always supercritical since the coefficient of $A_1|A_1|$ in \cref{e:gwBifEqub} is negative and the zero state is unstable for $\alpha>0$. 

\subsection{Numerical bifurcation analysis}\label{s:num}
We briefly illustrate and corroborate some of the analytical results by numerical continuation, cf.\ \citep{p2p,p2pbook}. Figure \ref{f:numbifiso} shows results for an isotropic case; Figure \ref{f:numbifaniso} for an anisotropic case. The supercritical bifurcations of GE and GWs occur as predicted and beyond the analysis, the branches extend to $\lb=0$, i.e.\ purely nonlinear bottom drag. Various bifurcation points are numerically detected along the branches, but are not shown.

Numerically, we can also study the stability of the bifurcating states. For isotropic backscatter, where steady and oscillatory modes are simultaneously critical, we find that the bifurcating solutions are all unstable (see Figure \ref{f:numbifiso}). This manifests already for purely $x$-dependent perturbations of the same wave number as the solutions, i.e.\ on the domain $[0,2\pi]$. Both the GE and GWs have a complex conjugate pair of unstable eigenvalues, where the unstable eigenfunction for GE has the shape of an GWs, and vice versa, suggesting that the instability stems from the interaction between GE and GWs. 

In Figure \ref{f:numbifaniso} we plot results for anisotropic backscatter, where GE bifurcate supercritically as predicted near $\lb=\lb_\crit=0.1$, and GWs at some smaller value of $\lb$.  
Interestingly, the bifurcating nonlinear GE appear to be spectrally stable against general 2D perturbations, at least  until $\lb\approx 0.01$. Here we compute the spectrum via a Floquet-Bloch transform in the $x$-direction which gives the Floquet-Bloch wave number parameter $\gamma_x\in[-\pi,\pi]$ and in the $y$-direction by Fourier transform with wave number $\gamma_y\in\R$. In Figure \ref{f:numbifaniso}(c,d) we plot the sidebands. Hence, it appears that combined backscatter and bottom drag can stabilize GE, so that backscatter not only induces the bifurcation of these, but also promotes a dynamic selection of geostrophically balanced states. 
\begin{figure}[t!]
\centering
\subfigure[]{\includegraphics[width=0.24\textwidth]{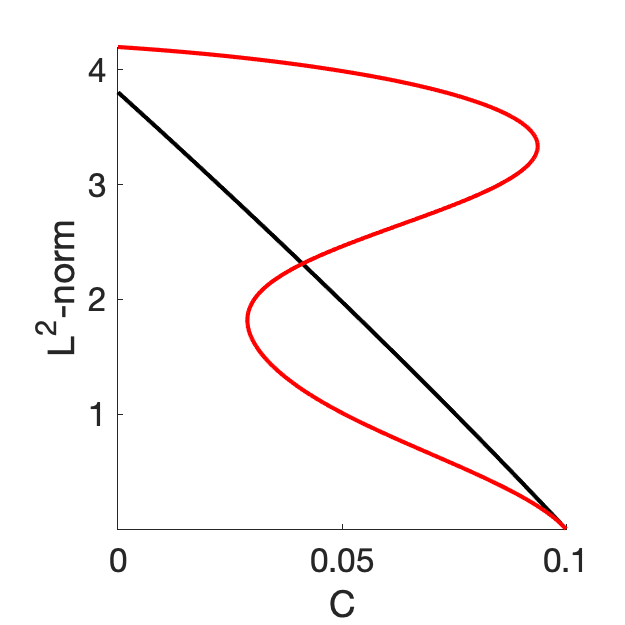}}
\hfil
\subfigure[]{\includegraphics[width=0.24\textwidth]{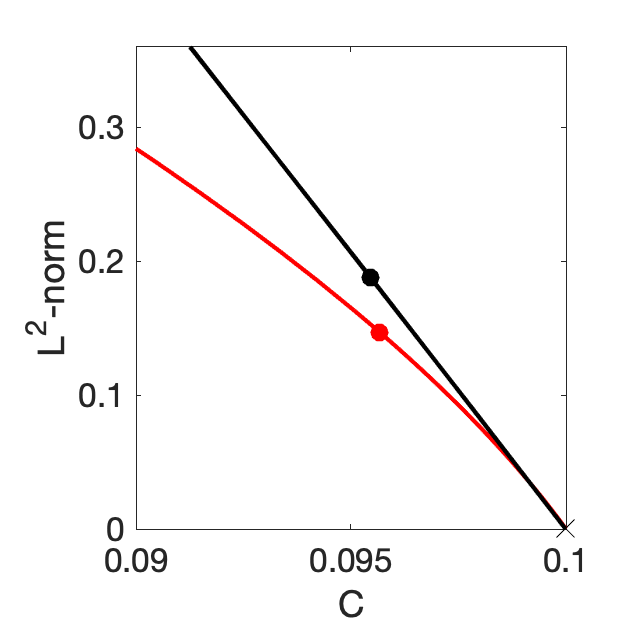}}
\hfil
\subfigure[]{\includegraphics[width=0.24\textwidth]{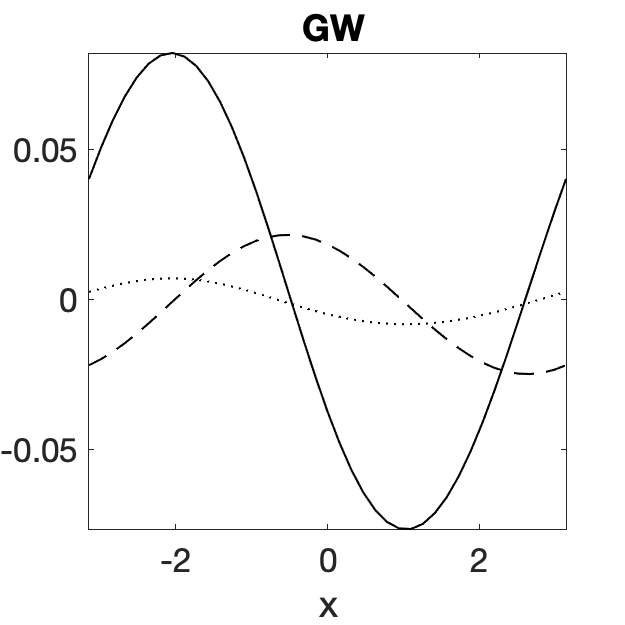}}
\hfil
\subfigure[]{\includegraphics[width=0.24\textwidth]{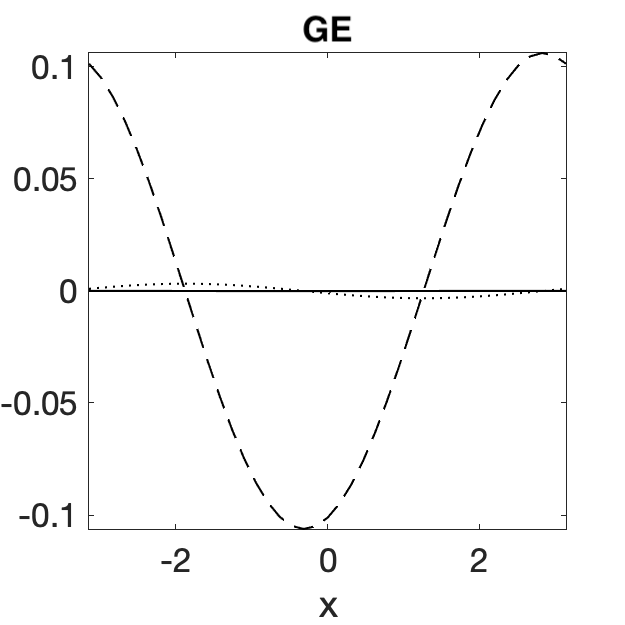}}
\caption{(Fig.~7 in \citep{PRY2023}.)
Numerical branches of nonlinear GE (black) and GWs (red) in the isotropic case $d_1=d_2=1$, $b_1=b_2=2$; (b) magnification of (a). Other parameters are $f=0.3$, $g=9.8$, $H_0=0.1$ so that $\lb_\crit=0.1$, and $Q=0.05$. 
Marked solutions are plotted in panels (c) and (d) with $u$ solid, $v$ dashed, $\eta$ dotted. 
}
\label{f:numbifiso}
\end{figure}

\begin{figure}[t!]
\centering
\subfigure[]{\includegraphics[width=0.24\textwidth]{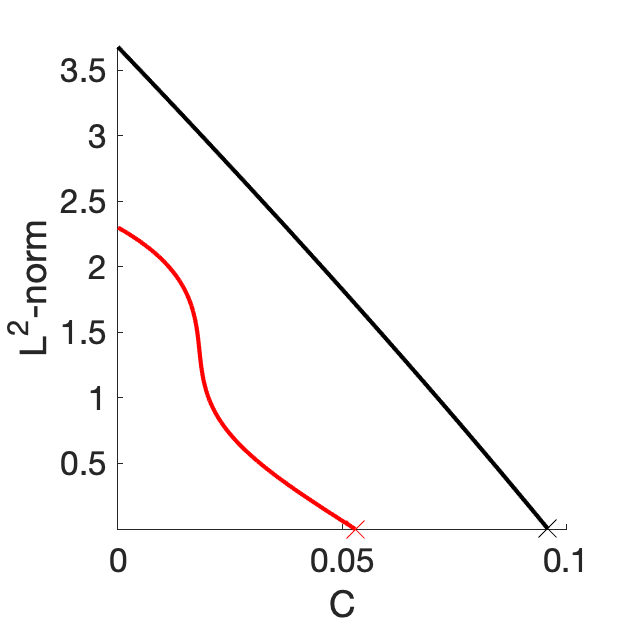}}
\hfil
\subfigure[]{\includegraphics[width=0.24\textwidth]{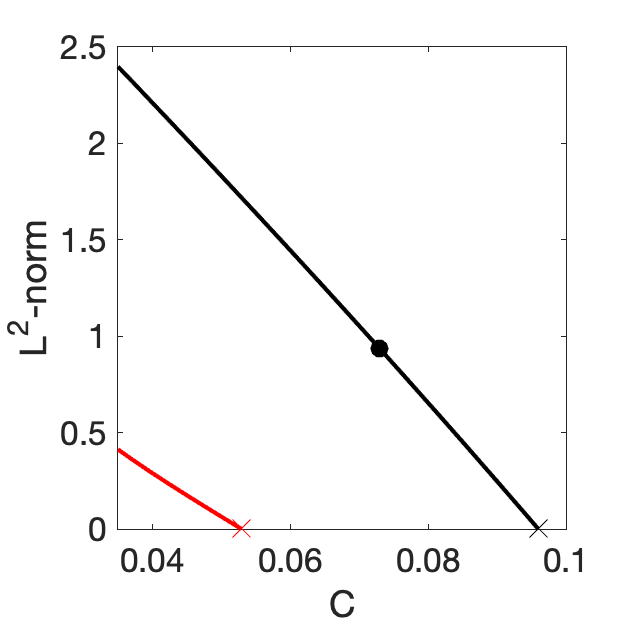}}
\hfil
\subfigure[]{\includegraphics[width=0.24\textwidth]{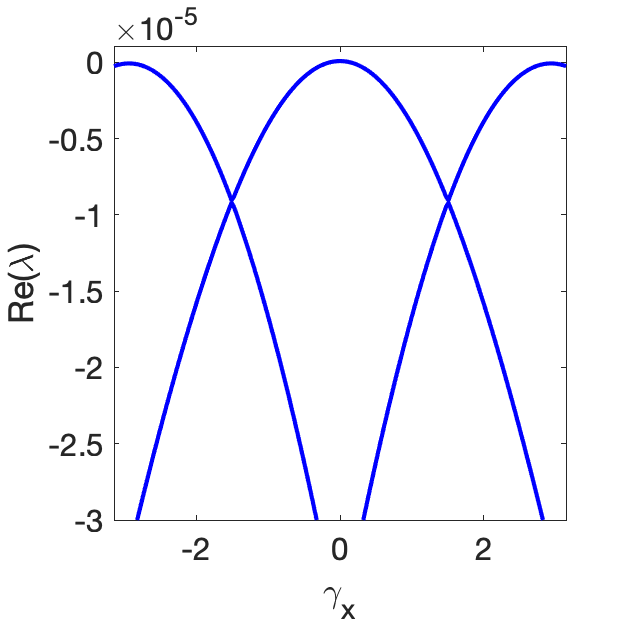}}
\hfil
\subfigure[]{\includegraphics[width=0.24\textwidth]{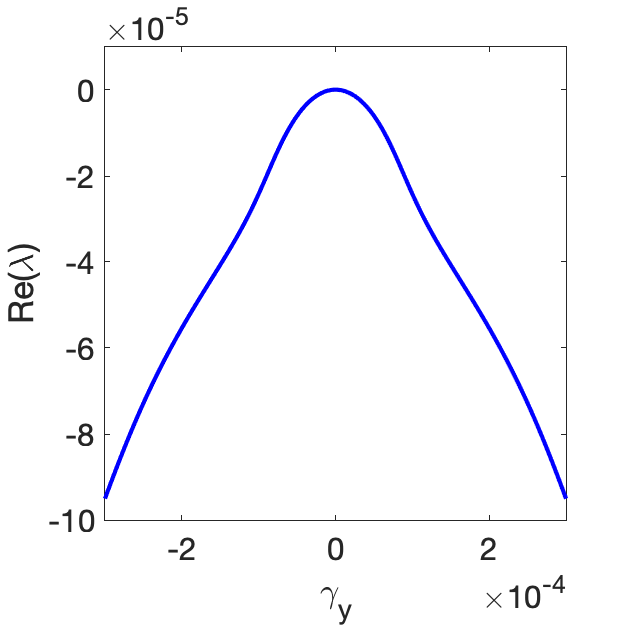}}
\caption{(Fig.~8 in \citep{PRY2023}.)
Numerical branches of nonlinear GE (black) and GWs (red) in the anisotropic case $d_1=1$, $d_2=1.04$, $b_1=1.5$, $b_2=2$; (b) magnification of (a). 
Other parameters are $f=0.3$, $g=9.8$, $H_0=0.1$ so that $\lb_\crit=0.1$, and $Q=0.05$. 
(c,d) relevant Floquet-Bloch spectra for $\gamma_x\in[-\pi,\pi]$, $\gamma_y\approx 0$ of the marked solution in (b).
}
\label{f:numbifaniso}
\end{figure}

\section{Discussion}\label{s:discussion}
Numerical kinetic energy backscatter schemes \citep{JH2014,DJKO2019,JDKO2019,juricke2020kinematic,Perezhogin20} aim to compensate for energy losses due to unresolved processes and unphysically high dissipation and viscosity needed for stable numerical simulation. We have reviewed the main results of \citep{PRY22,PRY2023} concerning the impact of such schemes in idealized settings via classes of explicit solutions to fluid equations augmented with backscatter on the continuum level. In addition to the Boussinesq and rotating shallow water equations, we have included the 3D primitive equations and, as a subsystem in all of these, the 2D Euler equations. Parameterizations more generally are common practice in geophysical and climate modeling to compensate unresolved effects \citep{TRRBook}, but little is understood about the impact mathematically.  

On the continuum level it is possible to study the modification of the dispersion relation of the trivial zero state by horizontal kinetic energy backscatter explicitly. As expected, on unbounded domains backscatter creates a band of exponentially growing large scale modes, while on the torus or other bounded domains, such destabilization occurs for sufficiently large domains or negative viscosity coefficient. Less immediately expected are certain explicit plane wave type solutions that are steady in the inviscid case, but grow exponentially and unboundedly with backscatter, highlighting undesired energy concentration through the backscatter. 
Due to superposition principles, such an unbounded instability is also induced for otherwise steady  plane wave type solutions themselves. We also reviewed results for the inclusion of linear and nonlinear bottom drag as an additional damping mechanism, which show that moderate linear damping does not suppress this phenomenon and nonlinear damping yields selection mechanisms by bifurcations and stabilization. 

As a new result, we have proven that for the 2D Euler equations unbounded exponential growth occurs in open regions of phase space, %and is thus quite robust, 
if the ratio of negative viscosity and hyperviscosity coefficients is not too large{\blue, thus preventing resonances}. This is inherited by the other models as relatively open sets on certain subspaces in phase space. From a more mathematical viewpoint, it would be interesting {\blue to see} whether this {\blue stability} persists for general backscatter coefficients, 
% that allow for resonances, 
and to extend this type of result to other equations. 

{\blue The fact} that backscatter induces some growth and destabilization is no surprise{\blue, but  that it} occurs as unbounded growth and {\blue in a} persistent {\blue manner} is not obvious. {\blue This may help explain o}bservations of persistent structures in some backscatter simulations \citep{JurickePersonal}, and difficulties to stably simulate with backscatter, e.g., \citep{Guan2022,juricke2020kinematic} and the references therein. Heuristically, {\blue persistent} growth is possible since the energy transfer from large to small scales is not strong enough {\blue to compensate the} energy pumped into the system on large scales.

\medskip
It would be interesting to study analytically and numerically in what way the undesired growth occurs in numerical discretizations. For fixed backscatter coefficients but decreasing grid size, the solutions of the discrete system are expected to track those of the continuum over increasingly long time intervals. Hence, the presence of unbounded growth in the limit is expected to readily imply large, but possibly finite growth for any fixed small grid size.  Notably, in the numerical backscatter scheme, the backscatter coefficients are scaled essentially proportional to the grid size. 
%It would be interesting to identify, for a given discretization scheme, the relation between the growth of flows and the grid-scaling of backscatter coefficients. 
{\blue For a given discretization scheme,} it would thus be interesting to identify the relation between the growth of flows and the grid-scaling of backscatter coefficients. 
This may provide another approach to energetic consistency of backscatter schemes. 

In contrast to the unbounded growth, the bifurcations that were found can be expected to persist under discretization, including stability on finite domains; stability of the bifurcating solutions for large domains is subtle and would be interesting to pursue further. We recall that the shallow water equation analysis is hampered by the lack of spectral gap, which does not appear in the viscous case and suggests lack of smoothing despite hyperviscosity. 

This is one motivation to take a more mathematical perspective, and consider well-posedness of PDE augmented with backscatter. In this chapter, we outlined some results in this direction with focus on the primitive equations, which are prominent in applications. Concerning this perspective, we note that the Fourier symbol of horizontal kinetic backscatter  in the horizontal momentum equation is that of the famous Kuramoto-Sivashinsky equations. However, the mathematical analysis differs compared to recent studies of this in higher space dimensions \citep{FM22,FSW22,KM2023}. %\JYx{Added a ref. here.}. 
Very recently \citep{KornTiti} prove well-posedness for another parameterization and suggest to use well-posedness as a criterion for sensible parameterizations. 

\bigskip
\textbf{Acknowledgement.} 

This paper is a contribution to the projects M1 and M2 of the Collaborative Research Centre TRR 181 ``Energy Transfers in Atmosphere and Ocean" funded by the Deutsche Forschungsgemeinschaft (DFG, German Research Foundation) under project number 274762653. The authors acknowledge the major contributions of Artur Prugger to the work we have reviewed. The authors thank Stephan Juricke and Marcel Oliver for discussions, and  Ekaterina Bagaeva for comments on \S\ref{s:intro} and \S\ref{s:modeling}. 
P.H. thanks Edriss Titi for the hospitality during a research visit and the inspiration for Theorem~\ref{thm:stabgrow}.

\bibliographystyle{apalike}
\bibliography{ReferencesGAFDpaper,ReferencesSIADSpaper,References,ReferencesIMApaper}

\begin{thebibliography}{}

\bibitem[Achatz, 2006]{Achatz06}
Achatz, U. (2006).
\newblock {\em Gravity-Wave Breakdown in a Rotating Boussinesq Fluid: Linear
  and Nonlinear Dynamics}.
\newblock Habilitation Thesis. University of Rostock.

\bibitem[Alexakis and Biferale, 2018]{AlexBifer2018}
Alexakis, A. and Biferale, L. (2018).
\newblock Cascades and transitions in turbulent flows.
\newblock {\em Physics Reports}, 767-769:1--101.
\newblock Cascades and transitions in turbulent flows.

\bibitem[Ambrose and Mazzucato, 2019]{AM2019}
Ambrose, D.~M. and Mazzucato, A.~L. (2019).
\newblock Global existence and analyticity for the 2{D}
  {K}uramoto-{S}ivashinsky equation.
\newblock {\em J. Dynam. Differential Equations}, 31(3):1525--1547.

\bibitem[Ambrose and Mazzucato, 2021]{AM2021}
Ambrose, D.~M. and Mazzucato, A.~L. (2021).
\newblock Global solutions of the two-dimensional {K}uramoto-{S}ivashinsky
  equation with a linearly growing mode in each direction.
\newblock {\em J. Nonlinear Sci.}, 31(6):96.

\bibitem[Arbic and Scott, 2008]{AS2008}
Arbic, B.~K. and Scott, R.~B. (2008).
\newblock On quadratic bottom drag, geostrophic turbulence, and oceanic
  mesoscale eddies.
\newblock {\em J. Phys. Oceanogr.}, 38(1):84--103.

\bibitem[Balmforth and Young, 2005]{BalmforthYoung2005}
Balmforth, N.~J. and Young, Y.-N. (2005).
\newblock Stratified {Kolmogorov flow. II}.
\newblock {\em J. Fluid Mech.}, 528:23--42.

\bibitem[Ben-Gal, 2009]{BenGal}
Ben-Gal, N. (2009).
\newblock {\em Grow-Up Solutions and Heteroclinics to Infinity for Scalar
  Parabolic PDEs}.
\newblock PhD thesis, Applied Mathematics Theses and Dissertations. Brown
  University and Free University of Berlin.

\bibitem[Cao and Titi, 2007]{CT07}
Cao, C. and Titi, E.~S. (2007).
\newblock Global well-posedness of the three-dimensional viscous primitive
  equations of large scale ocean and atmosphere dynamics.
\newblock {\em Ann. of Math. (2)}, 166(1):245--267.

\bibitem[Constantin et~al., 1988]{CFT88}
Constantin, P., Foias, C., and Temam, R. (1988).
\newblock On the dimension of the attractors in two-dimensional turbulence.
\newblock {\em Phys. D}, 30(3):284--296.

\bibitem[Coti~Zelati et~al., 2021]{CZDFM2021}
Coti~Zelati, M., Dolce, M., Feng, Y., and Mazzucato, A.~L. (2021).
\newblock Global existence for the two-dimensional {K}uramoto-{S}ivashinsky
  equation with a shear flow.
\newblock {\em J. Evol. Equ.}, 21(4):5079--5099.

\bibitem[Danilov et~al., 2019]{DJKO2019}
Danilov, S., Juricke, S., Kutsenko, A., and Oliver, M. (2019).
\newblock Toward consistent subgrid momentum closures in ocean models.
\newblock In Eden, C. and Iske, A., editors, {\em Energy Transfers in
  Atmosphere and Ocean}, pages 145--192. Springer, Cham.

\bibitem[Eden and Iske, 2019]{TRRBook}
Eden, C. and Iske, A., editors (2019).
\newblock {\em Energy transfers in atmosphere and ocean}, volume~1 of {\em
  Mathematics of Planet Earth}.
\newblock Springer, Cham.

\bibitem[Feng and Mazzucato, 2022]{FM22}
Feng, Y. and Mazzucato, A.~L. (2022).
\newblock Global existence for the two-dimensional {K}uramoto-{S}ivashinsky
  equation with advection.
\newblock {\em Comm. Partial Differential Equations}, 47(2):279--306.

\bibitem[Feng et~al., 2022]{FSW22}
Feng, Y., Shi, B., and Wang, W. (2022).
\newblock Dissipation enhancement of planar helical flows and applications to
  three-dimensional {Kuramoto-Sivashinsky} and {K}eller-{S}egel equations.
\newblock {\em J. Differential Equations}, 313:420--449.

\bibitem[Grooms, 2023]{Grooms}
Grooms, I. (2023).
\newblock Backscatter in energetically-constrained {Leith} parameterizations.
\newblock {\em Ocean Model.}, 186:102265.

\bibitem[Guan et~al., 2022]{Guan2022}
Guan, Y., Chattopadhyay, A., Subel, A., and Hassanzadeh, P. (2022).
\newblock Stable a posteriori {LES} of 2{D} turbulence using convolutional
  neural networks: {B}ackscattering analysis and generalization to higher {R}e
  via transfer learning.
\newblock {\em J. Comput. Phys.}, 458:111090.

\bibitem[Gustafson et~al., 2010]{GNT2010}
Gustafson, S., Nakanishi, K., and Tsai, T.-P. (2010).
\newblock Asymptotic stability, concentration, and oscillation in harmonic map
  heat-flow, {L}andau-{L}ifshitz, and {S}chr\"{o}dinger maps on {$\Bbb R^2$}.
\newblock {\em Comm. Math. Phys.}, 300(1):205--242.

\bibitem[Jansen et~al., 2019]{JansenEtAl2019}
Jansen, M.~F., Adcroft, A., Khani, S., and Kong, H. (2019).
\newblock Toward an energetically consistent, resolution aware parameterization
  of ocean mesoscale eddies.
\newblock {\em J. Adv. Model. Earth Syst.}, 11(8):2844--2860.

\bibitem[Jansen and Held, 2014]{JH2014}
Jansen, M.~F. and Held, I.~M. (2014).
\newblock Parameterizing subgrid-scale eddy effects using energetically
  consistent backscatter.
\newblock {\em Ocean Model.}, 80:36--48.

\bibitem[Juricke et~al., 2020]{juricke2020kinematic}
Juricke, S., Danilov, S., Koldunov, N., Oliver, M., Sein, D.~V., Sidorenko, D.,
  and Wang, Q. (2020).
\newblock A kinematic kinetic energy backscatter parametrization: From
  implementation to global ocean simulations.
\newblock {\em J. Adv. Model. Earth Syst.}, 12(12):e2020MS002175.

\bibitem[Juricke et~al., 2019]{JDKO2019}
Juricke, S., Danilov, S., Kutsenko, A., and Oliver, M. (2019).
\newblock Ocean kinetic energy backscatter parametrizations on unstructured
  grids: Impact on mesoscale turbulence in a channel.
\newblock {\em Ocean Model.}, 138:51--67.

\bibitem[Juricke and Koldunov, 2023]{JurickePersonal}
Juricke, S. and Koldunov, N. (2023).
\newblock Personal communication.

\bibitem[Kalogirou et~al., 2015]{Kalogirou15}
Kalogirou, A., Keaveny, E.~E., and Papageorgiou, D.~T. (2015).
\newblock An in-depth numerical study of the two-dimensional
  {K}uramoto-{S}ivashinsky equation.
\newblock {\em Proc. R. Soc. A.}, 471(2179):20140932.

\bibitem[Kloeden, 1985]{K85}
Kloeden, P.~E. (1985).
\newblock Global existence of classical solutions in the dissipative shallow
  water equations.
\newblock {\em SIAM J. Math. Anal.}, 16(2):301--315.

\bibitem[Kl\"ower et~al., 2018]{KJC2018}
Kl\"ower, M., Jansen, M.~F., Claus, M., Greatbatch, R.~J., and Thomsen, S.
  (2018).
\newblock Energy budget-based backscatter in a shallow water model of a double
  gyre basin.
\newblock {\em Ocean Model.}, 132:1--11.

\bibitem[Korn and Titi, 2023]{KornTiti}
Korn, P. and Titi, E.~S. (2023).
\newblock Global well-posedness of the primitive equations of large-scale ocean
  dynamics with the {G}ent-{M}c{W}illiams-{R}edi eddy parametrization model.

\bibitem[Kukavica and Massatt, 2023]{KM2023}
Kukavica, I. and Massatt, D. (2023).
\newblock On the global existence for the {K}uramoto-{S}ivashinsky equation.
\newblock {\em J. Dynam. Differential Equations}, 35(1):69--85.

\bibitem[Matthews and Cox, 2000a]{MC2000b}
Matthews, P.~C. and Cox, S.~M. (2000a).
\newblock One-dimensional pattern formation with {G}alilean invariance near a
  stationary bifurcation.
\newblock {\em Phys. Rev. E}, 62:R1473--R1476.

\bibitem[Matthews and Cox, 2000b]{MC2000}
Matthews, P.~C. and Cox, S.~M. (2000b).
\newblock Pattern formation with a conservation law.
\newblock {\em Nonlinearity}, 13(4):1293--1320.

\bibitem[Medjo, 2008]{M08}
Medjo, T.~T. (2008).
\newblock On strong solutions of the multi-layer quasi-geostrophic equations of
  the ocean.
\newblock {\em Nonlinear Analysis: Theory, Methods \& Applications},
  68(11):3550--3564.

\bibitem[Nave, 2008]{mitcode}
Nave, J.-C. (2008).
\newblock {M}atlab code `mit18336\_spectral\_ns2d.m'.
\newblock \url{https://math.mit.edu/~gs/cse/}.
\newblock Last checked on Jan~26, 2024.

\bibitem[Nicolaenko et~al., 1985]{NST85}
Nicolaenko, B., Scheurer, B., and Temam, R. (1985).
\newblock Some global dynamical properties of the {Kuramoto-Sivashinsky}
  equations: Nonlinear stability and attractors.
\newblock {\em Phys. D: Nonlinear Phenom.}, 16(2):155--183.

\bibitem[Olbers et~al., 2012]{OWEBook}
Olbers, D., Willebrand, J., and Eden, C. (2012).
\newblock {\em Ocean Dynamics}.
\newblock Springer Verlag Berlin, Berlin.

\bibitem[Pedlosky, 1987]{pedlosky1987geophysical}
Pedlosky, J. (1987).
\newblock {\em Geophysical Fluid Dynamics}.
\newblock Springer, New York, NY, 2nd edition.

\bibitem[Perezhogin, 2020]{Perezhogin20}
Perezhogin, P.~A. (2020).
\newblock Testing of kinetic energy backscatter parameterizations in the {NEMO}
  ocean model.
\newblock {\em Russian J. Numer. Anal. Math. Modelling}, 35(2):69--82.

\bibitem[Prugger and Rademacher, 2021]{PR2020}
Prugger, A. and Rademacher, J. D.~M. (2021).
\newblock Explicit superposed and forced plane wave generalized {B}eltrami
  flows.
\newblock {\em IMA J. Appl. Math.}, 86(4):761--784.

\bibitem[Prugger et~al., 2022]{PRY22}
Prugger, A., Rademacher, J. D.~M., and Yang, J. (2022).
\newblock Geophysical fluid models with simple energy backscatter: explicit
  flows and unbounded exponential growth.
\newblock {\em Geophys. Astrophys. Fluid Dynam.}, 116(5-6):374--410.

\bibitem[Prugger et~al., 2023]{PRY2023}
Prugger, A., Rademacher, J. D.~M., and Yang, J. (2023).
\newblock Rotating shallow water equations with bottom drag: Bifurcations and
  growth due to kinetic energy backscatter.
\newblock {\em SIAM J. Appl. Dyn. Syst.}, 22(3):2490--2526.

\bibitem[Rapha\"{e}l and Schweyer, 2013]{RS2013}
Rapha\"{e}l, P. and Schweyer, R. (2013).
\newblock Stable blowup dynamics for the 1-corotational energy critical
  harmonic heat flow.
\newblock {\em Comm. Pure Appl. Math.}, 66(3):414--480.

\bibitem[Temam, 1983]{T83}
Temam, R. (1983).
\newblock {\em Navier-{S}tokes equations and nonlinear functional analysis},
  volume~41 of {\em CBMS-NSF Regional Conference Series in Applied
  Mathematics}.
\newblock Society for Industrial and Applied Mathematics (SIAM), Philadelphia,
  PA.

\bibitem[Uecker, 2022]{p2pbook}
Uecker, H. (2022).
\newblock Continuation and bifurcation in nonlinear {PDEs} --algorithms,
  applications, and experiments.
\newblock {\em Jahresber. Dtsch. Math.-Ver.}, 124(1):43--80.

\bibitem[Uecker et~al., 2014]{p2p}
Uecker, H., Wetzel, D., and Rademacher, J. D.~M. (2014).
\newblock pde2path - a {Matlab} package for continuation and bifurcation in
  {2D} elliptic systems.
\newblock {\em Numer. Math. Theory Methods Appl.}, 7(1):58--106.

\bibitem[Vallis, 2017]{Vallis2017}
Vallis, G.~K. (2017).
\newblock {\em Atmospheric and Oceanic Fluid Dynamics: Fundamentals and
  Large-Scale Circulation}.
\newblock Cambridge University Press, 2 edition.

\bibitem[Wang, 1990]{Wang90}
Wang, C.~Y. (1990).
\newblock Exact solutions of the {N}avier-{S}tokes equations---the generalized
  {B}eltrami flows, review and extension.
\newblock {\em Acta Mech.}, 81(1-2):69--74.

\bibitem[Zeidler, 1990]{Zeidler90}
Zeidler, E. (1990).
\newblock {\em Nonlinear functional analysis and its applications. {II}/{B}}.
\newblock Springer-Verlag, New York.
\newblock Nonlinear monotone operators, Translated from the German by the
  author and Leo F. Boron.

\bibitem[Zeitlin, 2010]{Zeitlin}
Zeitlin, V. (2010).
\newblock Lagrangian dynamics of fronts, vortices and waves: Understanding the
  (semi-)~geostrophic adjustment.
\newblock In Flor, J.-B., editor, {\em Fronts, Waves and Vortices in
  Geophysical Flows}, pages 109--137. Springer Berlin Heidelberg.

\bibitem[Zurita-Gotor et~al., 2015]{ZuritaEtAl2015}
Zurita-Gotor, P., Held, I.~M., and Jansen, M.~F. (2015).
\newblock Kinetic energy-conserving hyperdiffusion can improve low resolution
  atmospheric models.
\newblock {\em J. Adv. Model. Earth Syst.}, 7(3):1117--1135.

\end{thebibliography}

\end{document}